\documentclass[a4paper]{amsart}
\usepackage{amssymb,latexsym}
\usepackage[utf8x]{inputenc}
\usepackage{amsmath,amsthm}
\usepackage{amsfonts,mathrsfs}
\usepackage{hyperref}
\usepackage{cleveref}
\usepackage{enumerate,units}
\usepackage[all]{xy}
\usepackage{graphicx,url}
   
\theoremstyle{plain}
\newtheorem{theorem}{Theorem}[section]
\newtheorem{corollary}[theorem]{Corollary}
\newtheorem{lemma}[theorem]{Lemma}
\newtheorem{proposition}[theorem]{Proposition}

\newtheorem{fact}[theorem]{Fact}
\newtheorem*{claim}{Claim}
\newtheorem*{theorem*}{Theorem}
\newtheorem*{context*}{Assumption $\diamondsuit$}

\theoremstyle{definition}
\newtheorem{definition}[theorem]{Definition}
\newtheorem{example}[theorem]{Example}

\theoremstyle{remark}
\newtheorem{remark}[theorem]{Remark}

\newtheorem*{question}{\textbf{Question}}
\newtheorem*{conjecture}{\textbf{Conjecture}}

\numberwithin{equation}{section}
\newcommand{\forkindep}[1][]{%
  \mathrel{
    \mathop{
      \vcenter{
        \hbox{\oalign{\noalign{\kern-.3ex}\hfil$\vert$\hfil\cr
              \noalign{\kern-.7ex}
              $\smile$\cr\noalign{\kern-.3ex}}}
      }
    }\displaylimits_{#1}
  }
}

\newenvironment{claimproof}[1][\proofname]
  {%
    \proof[#1]%
  }
  {%
    \endproof%
  }

\newcounter{step}                   
    {\hfill $\clubsuit$             
     \vspace{7pt}\par}

\makeatletter
\providecommand*{\cupdot}{%
  \mathbin{%
    \mathpalette\@cupdot{}%
  }%
}
\newcommand*{\@cupdot}[2]{%
  \ooalign{%
    $\m@th#1\cup$\cr
    \hidewidth$\m@th#1\cdot$\hidewidth
  }%
}
\makeatother

\newcommand{\reg}{\mathrm{reg}}
\newcommand{\dom}{\mathrm{Dom}}
\newcommand{\U}{\underline}
\newcommand{\Ba}{\textbf{a}}
\newcommand{\Bb}{\textbf{b}}
\newcommand{\kap}{\varkappa}
\newcommand{\Sh}{\text{Sh}}
\newcommand{\LSh}{\text{LSh}}
\newcommand{\RSh}{\text{RSh}}
\newcommand{\Cyc}{\text{Cyc}}
\newcommand{\E}{\mathrel{E}}
\newcommand{\R}{\mathrel{R}}
\newcommand{\D}{\mathrel{D}}

\DeclareMathOperator{\acl}{acl}
\DeclareMathOperator{\dcl}{dcl} 
\DeclareMathOperator{\tp}{tp}

\DeclareMathOperator{\Ur}{U}
\DeclareMathOperator{\Rg}{Range}
\DeclareMathOperator{\Img}{Im}

\begin{document}
\title{Infinite Stable Graphs With Large Chromatic Number}


\author{Yatir Halevi, Itay Kaplan, and Saharon Shelah}

\thanks{The first author would like to thanks the Israel Science Foundation for its support of this research (grant No. 181/16) and the Kreitman foundation fellowship. The second author would like to thank the Israel Science Foundation for its support of this research (grants no. 1533/14 and 1254/18). The third author would like to thank the Israel Science Foundation grant no: 1838/19 and the European Research Council grant 338821. Paper no. 1196 in the third author's  publication list.}

\email{yatirbe@post.bgu.ac.il}
\email{kaplan@math.huji.ac.il}
\email{shelah@math.huji.ac.il}
\keywords{chromatic number; stable graphs; Taylor's conjecture}
\subjclass[2010]{03C45; 05C15}

\begin{abstract}
We prove that if $G=(V,E)$ is an $\omega$-stable (respectively, superstable) graph with $\chi(G)>\aleph_0$ (respectively, $2^{\aleph_0}$) then $G$ contains all the finite subgraphs of the shift graph $\Sh_n(\omega)$ for some $n$. We prove a variant of this theorem for graphs interpretable in stationary stable theories. Furthermore, if $G$ is $\omega$-stable with $\Ur(G)\leq 2$ we prove that $n\leq 2$ suffices.
\end{abstract}

\maketitle
\section{Introduction}
The chromatic number $\chi(G)$ of a graph $G=(V,E)$ is the minimal cardinal $\kap$ for which the exists a vertex coloring with $\kap$ colors. There is a long history of structure theorems deriving from large chromatic number assumptions. For example if $\chi(G)\geq \aleph_1$ then $G$ must contain all finite bipartite graphs \cite[Corollary 5.6]{EH} and every sufficiently large odd circuit \cite[Theorem 3]{EHS}, \cite{thomassen}. See \cite{komjath} for more information.

In \cite[Problem 1.14]{taylor1}, Taylor asked what is the least cardinal $\kap$ such that every graph $G$ with $\chi(G)\geq \kap$ is elementary equivalent to graphs of arbitrarily large chromatic number. It is clear that such a minimal cardinal exists (see \cite[Theorem 1.13]{taylor1}). Taylor noted that necessarily $\kap\geq \aleph_1$. Nowadays, Taylor's conjecture is usually phrased in the following way (see \cite[Section 3]{komjath}).
\begin{conjecture}[Taylor's Conjecture]
For any graph $G$ with $\chi(G)\geq \aleph_1$ and cardinal $\kappa$ there exists a graph $H$ with $\chi(H)\geq \kappa$ such that $G$ and $H$ share the same finite subgraphs.
\end{conjecture}

For a caridnal $\kappa$ the shift graph $\Sh_n(\kappa)$ is the graph whose vertices are increasing $n$-tuples $s$ of ordinals less than $\kappa$, where we put an edge between $s$ and $t$ if for every $1\leq i\leq n-1$, $s(i)=t(i-1)$ or vice-versa. The shift graphs $\Sh_n(\kappa)$ have large chromatic numbers depending on $\kappa$, see Fact \ref{F:shft-largchr}. Erd\"os-Hajnal-Shelah \cite[Problem 2]{EHS} and Taylor \cite[Problem 43, page 508]{Taylorprob43} proposed the following strengthening of this conjecture. 
\begin{conjecture}[Strong Taylor's Conjecture]
For any graph $G$ with $\chi(G)\geq \aleph_1$ there exists an $n\in\mathbb{N}$ such that $G$ contains all finite subgraphs of $\Sh_n(\omega)$.
\end{conjecture}

Assuming the strong Taylor's Conjecture, if $\chi(G)\geq \aleph_1$ there exists an elementary extension $G\prec \mathcal{G}$ that has $\Sh_n(\beth_{n-1}(\kappa)^+)$ as a subgraph, and thus $\chi(\mathcal{G})\geq \kappa^+$, see Fact \ref{F:shft-largchr}. So the strong Taylor's conjecture implies Taylor's conjecture. It is known that Taylor's conjecture is consistently false and that a relaxation of Taylor's conjecture is consistently true, namely assuming that $\chi(G)\geq \aleph_2$ \cite{KS}. The strong Taylor's conjecture was refuted in \cite[Theorem 4]{HK}.

Since the (strong) Taylor's conjecture fails in general, one may wonder if it holds for a ``tame" class of graphs. Classification theory provides ``dividing lines" separating ``tame" and ``wild" classes of structures (and theories). These dividing lines are usually defined by requiring that a structure omits a certain class of (definable) combinatorial patterns. It is thus not surprising that restricting to such graphs will yield better combinatorial results.

An important instance of this phenomena is when tame=stable. Stable theories, which originated in the work of the third author in the 60s and 70s, is the most extensively studied class. Examples of stable theories include abelian groups, modules, algebraically closed fields, graph theoretic trees, or more generally superflat graphs \cite{PZ}. Stablility also had an impact in combinatorics, e.g. \cite{MS} and \cite{CPT} to name a few.

In this paper we prove variants of the strong Taylor's conjecture for some classes of stable graphs.

\begin{theorem*}
Let $G=(V,E)$ be a graph. If
\begin{enumerate}
\item $G$ is $\omega$-stable and $\chi(G)>\aleph_0$ or
\item $G$ is superstable and $\chi(G)>2^{\aleph_0}$ or
\item $G$ is interpretable in a stable structure, in which every type (over any set) is stationary, and $\chi(G)> \beth_2(\aleph_0)$
\end{enumerate}
then $G$ contains all finite subgraphs of $\Sh_n(\omega)$ for some $n\in \mathbb{N}$.

Furthermore, if $G$ is $\omega$-stable with $\chi(G)>\aleph_0$ and $\Ur(G)\leq 2$ then $n\leq 2$ suffices.
\end{theorem*}

Items $(1)$ and $(2)$ are Corollary \ref{C:superstable,omegastable}, $(3)$ is Corollary \ref{C:stationarystable} and the furthermore is Theorem \ref{T:Uleq2}.
%
%
%

The following remains open.
\begin{question}
\begin{enumerate}
\item What is the situation with general stable graphs?
\item Is it enough to assume $\chi(G)>\aleph_0$ in the above theorem?
\item What about other tameness assumptions, e.g. NIP or simplicity?
\end{enumerate}
\end{question}

\subsection*{Acknowledgments}
Section 6 is joint work with Elad Levi. We would like to thank him for allowing us to add these results. We thank Hrushovski for Proposition \ref{P:stationary-Udi}.

\section{Notation and Preliminaries}
We use fairly standard model theoretic terminology and notation, see for example \cite{TZ}. We use small latin letters $a,b,c$ for tuples and capital letters $A,B,C$ for sets. We also employ the standard model theoretic abuse of notation and write $a\in A$ even for tuples when the length of the tuple is immaterial or understood from context. When we write $a\equiv_A b$ we mean that $\tp(a/A)=\tp(b/A)$.

For any two sets $A$ and $J$, let $A^{\underline{J}}$ be the set of injective functions from $\gamma$ to $A$ (where the notation is taken from the falling factorial notation), and if $(A,<)$ and $(J,<)$  are both linearly ordered sets, let $(A^{\U J})_<$ be the subset of $A^{\U J}$ consisting of strictly increasing functions. If we want to emphasize the order on $J$ we will write $(A^{\U{(J,<)}})_<$. For an ordinal $\gamma$, we set $A^{<\U{\gamma}}:=\bigcup_{\alpha<\gamma}A^{\U \alpha}$. Throughout this paper, we interchangeably use sequence notation and function notation for elements of $A^{\U J}$, e.g. for $f\in A^{\U J}$, $f(i)=f_i$. For any sequence $\eta$ we denote by $\Rg(\eta)$ the underlying set of the sequence (i.e. its image). If $(A,<^A)$ and $(B,<^B)$ are linearly ordered sets, then the most significant coordinate of the lexicographic order on $A\times B$ is the first one.

By a \emph{graph} we mean a pair $G=(V_1,E)$ where $E\subseteq V^2$ is symmetric and irreflexive. A \emph{graph homomorphism} between $G_1=(V_1,E_1)$ and $G_2=(V_2,E_2)$ is a map $f:V_1\to V_2$ such that $f(e)\in E_2$ for every $e\in E_1$. If $f$ is injective we will say that $f$ embeds $G_1$ into $G_2$ a subgraph. If in addition we require that $f(e)\in E_2$ if and only if $e\in E_1$ we will say that $f$ embeds $G_1$ into $G_2$ as an induced subgraph.

\begin{definition}
Let $G=(V,E)$ be a graph.
\begin{enumerate}
\item For a cardinal $\kap$, a \emph{vertex coloring} (or just coloring) of size $\kap$ is a function $c:V\to \kap$ such that $x\E y$ implies $c(x)\neq c(y)$ for all $x.y\in V$.
\item The \emph{chromatic number} $\chi(G)$ is the minimal cardinality of a vertex coloring of $G$.
\end{enumerate}
\end{definition}

\begin{remark}
Note that for a graph $G=(V,\E)$ with $|V|\geq 2$, $\chi(G)=1$ if and only if $|E|=\emptyset$. 
\end{remark}

Here are some useful easy and well known properties of the chromatic number function of graphs (we provide proofs for the convenience of the reader).

\begin{lemma}\label{L:basic-prop-chi}
Let $G=(V,\E)$ be a graph.
\begin{enumerate}

\item If $V=\bigcup_{i\in I} V_i$ then $\chi(G)\leq \sum_{i\in I} \chi(V_i, E\restriction V_i)$.
\item If $E=\bigcup_{i\in I} \E_i$ (with the $E_i$ being symmetric)  then $\chi(G)\leq \prod_{i\in I}\chi(V,E_i)$.
\item If $\varphi:H\to G$ is a graph homomorphism then $\chi(H)\leq \chi(G)$.
\item If $\varphi:(H,E^H)\to (G,E^G)$ is a surjective graph homomorphism with $e\in E^H\iff \varphi(e)\in E^G$ then $\chi(H)=\chi(G)$.
\end{enumerate}
\end{lemma}
\begin{proof}
$(1)$ Let $c_i:V_i\to \kap_i$ be a coloring of $(V_i,E\restriction V_i)$. Define a coloring $c:V\to \bigcup \{\kap_i\times\{i\}:i\in I\}$ by choosing for any $v\in V$ an $i_v\in I$ such that $v\in V_{i_v}$ and setting $c(v)=(c_{i_v}(v),i_v)$.

$(2)$ Let $c_i:V_i\to \kap_i$ be a coloring of $(V,E_i)$. Define a coloring $c:V\to \prod_{i\in I}\kap_i$ by $c(v)(i)=c_i(v)$.

$(3)$ Write $G=(V^G,E^G)$ and $H=(V^H,E^H)$ and let $c:V^G\to \kap$ be a coloring of $G$. Define a coloring $c^\prime:V^H\to \kap$ of $H$ by $c^\prime(v)=c(f(v))$. 

$(4)$ Let $c:V^H \to \kap$ be a coloring. We define a coloring $c^\prime: V^G\to \kap$ by choosing for any element $v\in V^G$ an element $w\in \varphi^{-1}(v)$ and setting $c^\prime(v)=c(w)$. It is a legal coloring since if $v_1\E^G v_2$ then $w_1\E^H w_2$ for any $w_1\in \varphi^{-1}(v_1)$ and $w_2\in \varphi^{-1}(v_2)$.
\end{proof}

We will mainly be interested with the following so called ``Shift Graphs", first defined by Erd\"os-Hajnal in \cite{EH-shift}.

\begin{example}[Shift Graph]
For any finite number $1\leq r$ and any linearly ordered set $(A,<)$, let $\Sh_r(A)$, or $\Sh_r(A,<)$ if we want to emphasize the order, (the shift graph on $A$) be the following graph: its set of vertices is the set $(A^{\U r})_<$ of increasing $r$-tuples, $s_0,\dots,s_{r-1}$, and we put an edge between $s$ and $t$ if for every $1\leq i\leq r-1$, $s(i)=t(i-1)$, or vice-versa. It is an easy exercise to show that $\Sh_r(A)$ is a connected graph. If $r=1$ this gives $K_{A}$, the complete graph on $A$. 
%
%
%
%
\end{example}

\begin{example}[Symmetric Shift Graph]
Let $1\leq r$ be any natural number and $A$ any set. The \emph{symmetric shift graph} $\Sh_r^{sym}(A)$ is defined similarly as the shift graph but with set of vertices $A^{\U r}$ (set of distinct $r$-tuples).
Note that $\Sh_r(A)$ is an induced subgraph of $\Sh_r^{sym}(A)$ (and that for $r=1$ they are both the complete graph on $A$). 
%
Recall that $\beth_0(\kap):=\kap$ and $\beth_{k+1}(\kap):=2^{\beth_{k}(\kap)}$.

\begin{fact}\cite[Proof of Theorem 2]{EH-shift}\label{F:shft-largchr}
Let $2\leq r<\omega$ be a natural number and $\kap$ be a cardinal,  \[\chi\left(\Sh_r^{sym}(\beth_{r-1}\left(\kap\right))\right)\leq\kap\] and \[\chi\left(\Sh_r(\beth_{r-1}\left(\kap\right)^{+})\right)\geq\kap^{+}.\]
\end{fact}
\begin{proof}
We first show that $\chi\left(\Sh_r^{sym}(\beth_{r-1}\left(\kap\right))\right)\leq\kap$. The proof is by induction on $r\geq2$. Suppose $r=2$. Let $<$ be the lexicographical order on $2^{\kap}$. Let $Y_{1}=\left( \left(2^{\kap}\right)^{\U 2}\right)_<$ be the set of increasing pairs let $Y_{2}$ be the complement. By Lemma \ref{L:basic-prop-chi}(1) it is enough to show that $\chi\left(Sh_2^{sym}(2^{\kap})\upharpoonright Y_{1}\right)\leq\kap$, $\chi\left(\Sh_2^{sym}(2^{\kap})\upharpoonright Y_{2}\right)\leq\kap$. The proofs for $Y_{1}$ and $Y_{2}$ are similar so we prove it just for $Y_{1}$. 

Given $\left(x,y\right)\in Y_{1}$, let $c\left(x,y\right)=\min\{i<\kap: x\left(i\right)\neq y\left(i\right)\}$. Suppose that $x<y<z\in2^{\kap}$ are such that $c\left(x,y\right)=c\left(y,z\right)$. Then $x\wedge y=y\wedge z$ (where $x\wedge y=x\upharpoonright c\left(x,y\right)$ ). As $x<y$ it must be that $x\left(c\left(x,y\right)\right)=0$ and $y\left(c\left(x,y\right)\right)=1$, but then there is no room for $z\left(c\left(x,y\right)\right)$ — contradiction. 

Now suppose that the claim is true for r. By induction, there is a coloring $d:\Sh_r^{sym}(\beth_{r}\left(\kap\right))\to2^{\kap}$. Let $\psi:\Sh_{r+1}^{sym}(\beth_{r}\left(\kap\right))\to \Sh_2^{sym}(2^{\kap})$ be the following homomorphism. Given $u=\left(u_{0},\ldots,u_{r}\right)\in \Sh_{r+1}^{sym}(\beth_{r}\left(\kap\right))$, let $\psi\left(u\right)=\left(d\left(u_{0},\ldots,u_{r-1}\right),d\left(u_{1},\ldots,u_{r}\right)\right)$. Note that by choice of $d$, \[d\left(u_{0},\ldots,u_{r-1}\right)\neq d\left(u_{1},\ldots,u_{r}\right).\] In addition, if $u$ and $v$ are connected in $\Sh_{r+1}^{sym}(\beth_{r}\left(\kap\right))$, then easily $\psi\left(u\right)$ and $\psi\left(v\right)$ are distinct (because if not, then $\psi\left(u\right)_{0}=\psi\left(v\right)_{0}=\psi\left(u\right)_{1}$ contradiction) and connected in $\Sh_2^{sym}(2^{\kap})$. Hence we are done by Lemma \ref{L:basic-prop-chi}(3). 

As for the second inequality, let $c:\Sh_r(\beth_{r-1}\left(\kap\right)^{+})\to \kap$ be a coloring. The coloring $c$ induces a coloring on $\Sh_r(\beth_{r-1}\left(\kap\right)^{+})$, by Erd\"os-Rado, there is a subset $U\subseteq \beth_{r-1}\left(\kap\right)^{+}$ of cardinality $\kap^+$ such that $c\restriction [U]^r$ is constant, i.e. every $r$-tuple of increasing elements from $U$ is colored by the same color.  As a consequence, there cannot be an edge between any $u,v\in [U]^r$. Indeed, let $u\in [U]^r$ be any element. Let $v\in [U]^r$ be defined by $v(i)=u(i+1)$ for $0\leq i<r-1$ and $v(r-1)=u(1)$. They are obviously connected by an edge.
\end{proof}

\end{example}

\section{Embedding a Shift Graph}
The aim of this section is to present some general assumptions on a graph $G$ that will imply that $G$ contains the finite subgraphs of some shift graph.
\subsection{Reducing Injective Homomorphisms to Homomorphisms}
As a first result we prove the following, probably well known, proposition. By Lemma \ref{L:basic-prop-chi}(3), if there is a homomorphism $\varphi:H\to G$ then $\chi(H)\leq \chi(G)$. In particular, if $H$ is a shift graph then there are elementary extensions of $G$ with arbitrary large chromatic numbers. Indeed, one may take elementary extensions of the structure $(H,G,\varphi)$ and apply Fact \ref{F:shft-largchr}.

\begin{fact}\cite[Theorem 1]{ER-comb}\label{F:ER-comb}
Let $R$ be an equivalence relation on $(\omega^{\U n})_<$. Then there exists an infinite subset $N\subseteq \omega$ and $0\leq i_1<\dots<i_m\leq n-1$ such that for $\bar a,\bar b\in (N^{\U n})_<$,
\[\bar a\mathrel{R} \bar b \iff \bigwedge_{j=1}^m a_{i_j}=b_{i_j}.\]
\end{fact}

\begin{proposition}\label{P:homomorphism-is-enough}
Let $G=(V,E)$ be a graph and assume there exists an homomorphism of graphs $t:\Sh_k(\omega)\to G$. Then there exists $n\leq k$, such that 
\begin{itemize}
\item[($\dagger$)] $G$ contains all finite subgraphs of $\Sh_n(\omega)$.
\end{itemize}

Consequently, if $H$ is a graph that contains all finite subgraphs of $\Sh_k(\omega)$, for some $k$, and $t:H\to G$ is a homomorphism of graphs, then there exists some $n\leq k$ such that $G$ satisfies ($\dagger$).
\end{proposition}
\begin{proof}
Assume that $t=t(x_0,\dots,x_{k-1})$. The relation $t(\bar a)=t(\bar b)$ for $\bar a,\bar b\in (\omega^{\U k})_<$, is an equivalence relation on $(\omega^{\U k})_<$. By Fact \ref{F:ER-comb}, there exists an infinite subset $N\subseteq \omega$ and $0\leq i_1<\dots<i_m\leq k-1$ such that for $\bar a,\bar b\in (N^{\U k})_<$
\[t(\bar a)=t(\bar b) \iff \bigwedge_{j=1}^m a_{i_j}=b_{i_j} \tag{$\dagger\dagger$}.\]
Note that $m\geq 1$ since otherwise $t(\bar a)=t(\bar b)$ for any $\bar a$ and $\bar b$, but this is impossible since there are $\bar a,\bar b\in (N^{\U k})_<$ that are connected by an edge.

Let $S=\{i_1,\dots, i_m\}$. There exists a unique set $I\subseteq \{1,\dots,m\}$ and a unique sequence of natural numbers $\bar n=\langle n_j:j\in I\rangle$ such that $S=\bigcup_{j\in I}[i_j, i_j+n_j]$ and each interval $[i_j, i_j+n_j]$ is maximal with respect to containment.

Consider the first-order structure $M=((N,<),G=(V,E), t:(N^{\U k})_<\to G)$. Since $(\dagger)$ and ($\dagger\dagger$) are elementary properties, replacing $M$ by an elementary extension, we may assume that $(N,<)=(I\times\mathbb{Z},<_{\text{lex}})$.

We define an injective homomorphism $\Sh_{n+1}(\omega)\to G$, where $n=\max_{j\in I}\{n_j\}$. For any $f\in (\omega^{\U{n+1}})_<$ we associate $\psi_f\in V$. For that we first define $\eta_f\in ((I\times\mathbb{Z})^{\U k})_<$ and then set $\psi_f=t(\eta_f)$. For any $j\in I$ and $0\leq r\leq n_j$ we define
\[\eta_f(i_j+r)=(j,f(r)).\]
For $0\leq i\leq k-1$ with $i\notin S$, $(*)$ we may set $\eta_f(i)$ any way we want provided $\eta_f$ is increasing, which we can since we have copies of $\mathbb{Z}$. Note that this is will not influence $t(\eta_f)$ by ($\dagger\dagger$).

We check that $f\mapsto \psi_f$ is an injective homomorphism. Injectivity: if $t(\eta_f)=t(\eta_g)$ then by ($\dagger\dagger$), $\eta_f(i)=\eta_g(i)$ for all $i\in S$. In particular for $j\in I$ with $n_j=n$ and for any $0\leq r\leq n$, $f(r)=g(r)$, as needed.

Homomorphism: let $f,g\in \Sh_{n+1}(\omega)$ and assume without loss of generality that for every  $1\leq r\leq n$, $f(r)=g(r-1)$ and in case $n=0$ assume that $f(0)<g(0)$. 
Using $(*)$, we may assume that $\eta_f(i)=\eta_g(i-1)$ for $1\leq i\leq k-1$. 
Indeed, consider the following modification on $\eta_g$. For every $j\in I$ and $0 \leq r\leq n_j$ keep $\eta_g(i_j+r)$ as before, and if $i_j>0$ then set
\[\eta_g(i_j-1)=(j,f(0)).\]
Note that if $n\neq 0$ then
\[\eta_g(i_j-1)=(j,f(0))<(j,f(1))=(j,g(0))=\eta_g(i_j)\]
and if $n=0$ then
\[\eta_g(i_j-1)=(j,f(0))<(j,g(0))=\eta_g(i_j).\]

Hence $\eta_g$ restricted to $S\cup \{i_j-1: j\in I,\, i_j>0\}$ is increasing. For any other $i\in I$ set $\eta_g(i)$ any way we want provided $\eta_g$ is increasing.

Define $\eta\in ((I\times\mathbb{Z})^{\U k})_<$ by 
\[\eta(i)=\eta_g(i-1)\]
for all $1\leq i\leq k-1$. If there exists $j\in I$ with $i_j=0$ define $\eta(0)=(0,f(0))$. Note that then if $n\neq 0$ then $\eta(0)=(0,f(0))<(0,f(1))=(0,g(0))=\eta_g(0)=\eta(1)$ and if $n=0$ then $\eta(0)=(0,f(0))<(0,g(0))=\eta_g(0)=\eta(1)$. Otherwise define $\eta(0)$ to be a new element smaller than any element we have encountered in $\eta_g$. If we show that $\eta(i)=\eta_f(i)$ for all $i\in S$ then this would imply that $t(\eta)=t(\eta_f)$. Since $\eta$  and $\eta_g$ are connected by an edge and $t$ is a homomorphism it follows that $\psi_f=t(\eta_f)$ and $\psi_g=t(\eta_g)$ are connected by an edge, as required.

Let $j\in I$ and $0\leq r\leq n_j$. If $r\geq 1$ then
\[\eta_f(i_j +r)=(j,f(r))=(j,g(r-1))=\eta_g(i_j +r-1)=\eta (i_j+r).\]
If $r=0$ and $i_j>0$ then
\[\eta_f(i_j)=(j,f(0))=\eta_g(i_j-1)=\eta(i_j).\]
Finally, if $r=0$ and $i_j=0$ then
\[\eta_f(0)=(0,f(0))=\eta(0).\]
%
%

As for the ``consequently" part, consider $(H,t,G)$ as a first order structure. In an elementary extension $(H,t,G)\prec (\mathcal{H},t,\mathcal{G})$, $\mathcal{H}$ contains $\Sh_n(\omega)$ as a subgraph. Restricting $t$ to $\Sh_n(\omega)$ and applying the above, $\mathcal{G}$ contains all finite subgraphs of $\Sh_n(\omega)$ for some $n\leq k$. As a result, so does $G$.
\end{proof}

\subsection{Variants of the Shift Graph}
Let $A$ and $J$ be two (possibly linearly ordered) sets.

\begin{definition}\label{D:f_a,b}
For any $\bar a,\bar b\in A^{\U J}$ (respectively, $(A^{\U J})_<$), let $f_{\bar a,\bar b}= \{(i,j)\in J\times J: a_i=b_j\}$.
\end{definition}

Since the tuples $\bar a $ and $\bar b$ are without repetitions, $f_{\bar a,\bar b}$ is a (possibly empty) injective partial function. If $\bar a, \bar b\in (A^{\U J})_<$ then $f_{\bar a,\bar b}$ is order-preserving, i.e. for all $i<j\in \dom(f)$, $f(i)<f(j)$.

\begin{definition}\label{D:E_f,D_f}
Let $Id\neq f\subseteq J\times J$ be a partial function. We define a graph $E^A_f$ and a directed graph $D_f^A$ on $A^{\U J}$:
\begin{list}{•}{}
\item $\bar a \E^A_f \bar b \iff f_{\bar a,\bar b}=f \vee f_{\bar b,\bar a}=f$
\item  $\bar a \D^A_f \bar b \iff f_{\bar a,\bar b}=f.$
\end{list} 
Similarly for $(A^{\U J})_<$. We omit $A$ from $E_f^A$ and $D_f^A$ when it is clear from the context.
\end{definition}

\begin{remark}
We required $f\neq Id$ in order to ensure irreflexivity.

A homomomorphism between directed graphs is map preserving the directed graph relation.

Since the symmetric closure of the relation $D_f$ is exactly $E_f$, any homomorphism of directed graphs $(A_1^{\U{J_1}}, D_{f_1})\to (A_2^{\U{J_2}}, D_{f_2})$ is also a homomorphism of graphs $(A_1^{\U{J_1}}, E_{f_1})\to (A_2^{\U{J_2}}, E_{f_2})$, and similarly in the ordered case. 
\end{remark}

\begin{example}
When $J=n$ and $f=\{(i,i-1):1\leq i\leq n-1\}$, $(A^{\U n},E_f)$ is exactly $\Sh^{sym}_n(A)$ and $((A^{\U n})_<,E_f)$ is exactly $\Sh_n(A)$.
\end{example}

\begin{definition}
Let $\LSh_n(A)=((A^{\U n})_<,D_f)$, where $f=\{(i,i-1): 1\leq i\leq n-1\}$, and $\RSh_n(A)=((A^{\U n})_<,D_f)$, where $f=\{(i-1,i): 1\leq i\leq n-1\}$.
\end{definition}

\begin{lemma}\label{L:embedding-intertwined}
Let $(J,<)$ be a finite linearly ordered set and $Id\neq f\subseteq J\times J$ a non-empty partial function. Assume that 
\begin{enumerate}
\item $f$ is order preserving, i.e. for all $i<j\in \dom(f)$, $f(i)<f(j)$,
\item all orbits in $f$ are \emph{increasing}, i.e. for all $i\in \dom(f)$, $i<f(i).$ \end{enumerate}
Let $G=((\mathbb{Q}^{\U J})_<, D_f)$ be the directed graph structure defined by $f$.
Then for any large enough $k\in \mathbb{N}$ there exists a homomorphism of directed graphs $\varphi: \LSh_k(\omega)\to ((\mathbb{Q}^{\U J})_<, D_f)$.
%

\end{lemma}
\begin{proof}


Let $A$ be the ordinal $\omega^\omega$, seen as a substructure of $\mathbb{Q}$. As $J$ is finite, we may assume that $(J,<)\subseteq (\mathbb{Q},<)$.

\begin{claim}
There is no harm in replacing $J$ by $\dom(f)\cup \Rg(f)$ and $\mathbb{Q}$ by $A$.
\end{claim}
\begin{claimproof}
Inductively, for every $u\in \LSh_k(\omega)$ choose a dense subset $Q_u\subseteq \mathbb{Q}$ such that for every $u\neq v\in \LSh_k(\omega)$, $Q_u\cap Q_v=\emptyset$ and $Q_u\cap \omega^\omega=\emptyset$. 

Now, let $\widehat J=\dom(f)\cup \Rg(f)$ and assume we have a homomorphism $\varphi: \LSh_k(\omega)\to ((\omega^\omega)^{\U{\widehat J}})_<,D_f)$. For each $u\in \LSh_k(\omega)$, extending $\varphi(u)$ to an increasing $J$-tuple of elements from $\mathbb{Q}$ by adding elements from $Q_u$, defines a map $\varphi^\prime: \LSh_k(\omega)\to ((\mathbb{Q}^{\U{J}})_<,D_f)$. Since $\widehat J=\dom(f)\cup \Rg(f)$ and passing from $\varphi(u)$ to $\varphi^\prime (u)$ adds only new elements, $\varphi^\prime$ is a homomorphism of directed graphs. 

\end{claimproof}

Let $I=\dom(f)\setminus \Rg(f)$. We prove by induction on $|I|$ that for any large enough $k$ there exists a homomorphism $g:\LSh_k(\omega)\to G=((A^{\U J})_<, D_f)$, where $J=\dom(f)\cup \Rg(f)$ is nonempty. 

For any $\beta\in I$ let $n_\beta$ be the maximal natural number $n\geq 1$ such that $f^{n-1}(\beta)\in \dom (f)$.  Note that \[\dom (f)=\bigcup_{\beta \in I} \{\beta,\dots, f^{n_\beta-1}(\beta)\}\]
and that 
 \[J=\dom (f)\cup\Rg(f)=\bigcup_{\beta \in I} \{\beta,\dots, f^{n_\beta}(\beta)\}.\]
%
Let $\beta_0$ be the minimal element of $I$. Note that $\beta_0$ is also the minimal element of $J$.

\begin{claim}
%
There exist $J\subseteq \widetilde J\subseteq \mathbb{Q}$ and $f\subseteq \widetilde f\subseteq  \widetilde J\times\widetilde J$ such that 
\begin{list}{•}{}
\item $(1),(2)$ of the lemma hold for $\widetilde J$ and $\widetilde f$,
\item $\widetilde f\cap  (J\times J)=f$,
\item $\dom(\widetilde f)\setminus \Rg(\widetilde f)=I$,
\item $\min \widetilde J=\beta_0$ and that
\item $(\star)$ letting $\widetilde n_{\beta_0}$ be the maximal natural number $n\geq 1$ such that $f^{n-1}(\beta_0)\in \dom(\widetilde f)$,  $f^{\widetilde n_{\beta_0}}(\beta_0)=\max \widetilde J$.

\end{list}
  
\end{claim}
\begin{claimproof}

If $(\star)$ holds for $f$ and $J$, we are done. 
Otherwise, let $i\in \Rg(f)$ be minimal such that $f^{n_{\beta_0}}(\beta_0)<i$ and let $j$ be such that $f(j)=i$. Either $j<f^{n_{\beta_0}}(\beta_0)$ or $f^{n_{\beta_0}}(\beta_0)<j$. If the former happens let $f^\prime=f$, so assume it is the latter, i.e. that $f^{n_{\beta_0}}(\beta_0)<j<i$ (so $j\in\dom(f)\setminus \Rg(f)$). Let $f^\prime=f\cup\{(f^{n_{\beta_0}}(\beta_0),y)\}$ for some $j<y\in \mathbb{Q}\setminus J$ which satisfies $y<x$ for all $j<x\in J$ and $J^\prime=J\cup \{y\}$. It is still order preserving and still has increasing orbits. Let $n_{\beta_0}^\prime$ be as in $(\star)$. 	

In either case, we have that $j<(f^\prime)^{n^\prime_{\beta_0}}(\beta_0)<i$. Since $f^\prime$ is order preserving we may extend it to an automorphism $\sigma$ of $\mathbb{Q}$ and thus $i=\sigma(j)<\sigma((f^\prime)^{n^\prime_{\beta_0}}(\beta_0))$. By careful adjustments we may assume that $\sigma((f^\prime)^{n^\prime_{\beta_0}}(\beta_0))\notin J$. Let $f^{\prime\prime}=f^\prime\cup\{((f^\prime)^{n^\prime_{\beta_0}}(\beta_0),\sigma((f^\prime)^{n^\prime_{\beta_0}}(\beta_0)))\}$ and let $J^{\prime\prime}=J\cup \{\sigma((f^\prime)^{n^\prime_{\beta_0}}(\beta_0))\}$. It is still order preserving and still has increasing orbits. 

Note that, \[|\{i\in \Rg(f^{\prime\prime}):i<(f^{\prime\prime})^{n^{\prime \prime}_{\beta_0}}(\beta_0)\}|<|\{i\in \Rg(f):i<f^{n_{\beta_0}}(\beta_0)\}|,\] where $n^{\prime \prime}_{\beta_0}$ is defined as is $(\star)$.
 
Continue doing this until $\Rg(f)$ is exhausted. Let $\widetilde f$ be the end function and let $\widetilde J=J\cup\dom(\widetilde f)\cup \Rg(\widetilde f)$. 

\end{claimproof} 

As a consequence of the claim we may assume that $(\star)$ holds for $f$ and $J$. Indeed, assume we found a homomorphism $\varphi:\LSh_k(\omega)\to ((A^{\U{\widetilde J}})_<,D_{\widetilde f})$, for some $k$. It is routine to check that composing $\varphi$ with the map \[((A^{\U{\widetilde J}})_<,D_{\widetilde f})\to ((A^{\U{J}})_<,D_{f})\] induced by the projection map $\pi:(A^{\U{\widetilde J}})_<\to (A^{\U{J}})_<$ is indeed a homomorphism (this uses the second bullet in the claim above). 

Let $J^\prime=J\setminus \{\beta_0,\dots, f^{n_{\beta_0}}(\beta_0)\}$. If $J^\prime=\emptyset$ let $g_k$ be the empty function for all $k\in \mathbb{N}$. Otherwise, by induction there exists $l\in \mathbb{N}$ such that for all $k\geq l$ there is a homomorphism $g_k:\LSh_{k}(\omega)\to ((A^{\U{J^\prime}})_<,D_{f\cap (J^\prime\times J^\prime)})$.  Let $k> \max\{n_\beta+1:\beta\in I\}\cup \{l\}$ and set some order isomorphism $\phi:\omega\times (A\cup\{-1\})\to A$, where $-1$ is a new element which is smaller than any element of $A$  (recall that $A=\omega^\omega$).

We construct a homomorphism mapping $\mu\in \LSh_k(\omega)$ to $\psi_\mu\in G$. Let $\mu \in \LSh_k(\omega)$. For any $0\leq h\leq n_{\beta_0}$ we define
\[\psi_\mu(f^h(\beta_0))=\phi(\mu(h),-1).\]
For any $\beta\in I$, with $\beta\neq \beta_0$, and $0\leq h\leq n_\beta$ we define
\[\psi_\mu(f^{h}(\beta))=\phi(\mu(\tilde h),g_k(\mu)(f^{h}(\beta))),\]
for $0\leq \tilde h\leq n_{\beta_0}$ maximal satisfying $f^{h}(\beta)>f^{\widetilde h}(\beta_0)$, which exists by minimlaity of $\beta_0$.

We check that $\psi_\mu$ is increasing and that $\mu\mapsto \psi_\mu$ is a homomorphism.

To show that $\psi_\mu$ is increasing, suppose $f^{h_1}(\beta_1)<f^{h_2}(\beta_2)\in J$ and go over the different possibilities. Note that we use $-1$ in the case when $\beta_1=\beta_0, \beta_2\neq \beta_0$ and $\widetilde h_2=h_1$.
%


We show that $\mu\mapsto \psi_\mu$ is a homomorphism. Suppose that $\mu,\nu\in \LSh_k(\omega)$ are such that $\nu(n+1)=\mu(n)$ for all $0\leq n<k-1$. We need to check that $f(i)=j$ if and only if $\psi_\mu(i)=\psi_\nu (j)$.

Assume that $f(i)=j$ (so $i\in \dom(f)$). Suppose that $i=f^h(\beta_0)$ for some $0\leq h<n_{\beta_0}$, so $j=f^{h+1}(\beta_0)$. Then
\[\psi_\mu(i)=\phi(\mu (h),-1)=\phi(\nu(h+1),-1)=\psi_\nu(f^{h+1}(\beta_0)).\]
Now suppose that $i=f^{h}(\beta)$ for some $\beta\neq \beta_0$ and $0\leq h<n_\beta$, so $j=f^{h+1}(\beta)$. Let $\widetilde h$ be maximal such that $f^{\widetilde h}(\beta_0)<f^h(\beta)$. Note that by the claim, $\widetilde h<n_{\beta_0}$. It follows that $f^{h+1}(\beta)$ is defined and $f^{\widetilde h+1}(\beta_0)<f^{h+1}(\beta)$. On the other hand, it can not be that $f^{\widetilde h+2}(\beta_0)<f^{h+1}(\beta)$ (again, $f^{\widetilde h+2}(\beta_0)$ is defined by the claim) for then we would have $f^{\widetilde h+1}(\beta_0)<f^h(\beta)$, contradicting the maximality of $\widetilde h$. It follows that $\widetilde h+1=\widetilde{h+1}$. Since $g_k$ is a homomorphism,
\[\psi_\mu(i)=\phi(\mu(\widetilde h),g_k(\mu)(f^h(\beta)))=\phi(\nu(\widetilde h+1),g_k(\nu)(f^{h+1}(\beta))=\psi_\nu(j).\]

Now assume that $\psi_\mu(i)=\psi_\nu(j)$. If $i=f^{h}(\beta_0)$ for some $0\leq h\leq n_{\beta_0}$ then $j=f^{h^\prime}(\beta_0)$ for some $h^\prime$ (since $\psi_\mu(j)$ has the form $\phi(-,-1)$) and so $\mu(h)=\nu(h^\prime)$. By the choice of the $k$, $h+1<k$ and consequently $\mu(h)=\nu(h+1)=\nu(h^\prime)$ so $h^\prime=h+1$ and $f(i)=j$. 


Suppose $i=f^{h}(\beta)$ for some $\beta\neq \beta_0$ and $0\leq h\leq n_\beta$. As this is encoded by $\phi$, by the assumption necessarily $j=f^{h^\prime}(\beta^\prime)$ for some $\beta^\prime\neq \beta_0$ and $0\leq h^\prime\leq n_{\beta ^\prime}$. Let $\widetilde h$ be maximal such that $f^{\widetilde h}(\beta_0)<i$ and $\widetilde h^\prime$ maximal such that $f^{\widetilde h^\prime}(\beta_0)<j$.  So $\psi_\mu(i)=\phi(\mu(\widetilde h),g_k(\mu)(i))$ and $\psi_\nu(j)=\phi(\nu(\widetilde h^\prime),g_k(\nu)(j))$. It follows that $g_k(\nu)(j)=g_k(\mu)(i)$ and we are done by the choice of $g_k$ since $i,j\in J^\prime$.
%
\end{proof}

Before continuing to the main proposition, as auxiliary results, we calculate the chromatic number of some (well known) graphs.

\begin{example}[Symmetric Cyclic Graph]\label{E:sym-cyclic}
Let $r>1$ be a natural number. Let $\Cyc_r^{sym}(A)$ be the graph on $A^{\U r}$ with an edge between $(a_0,\dots,a_{r-1})$ and $(b_0,\dots,b_{r-1})$ if $a_0=b_1,\dots, a_{r-2}=b_{r-1}, a_{r-1}=b_0$ (or vice-versa). We thus have a graph homomorphism $\Cyc_r^{sym}(A)\to \Sh_r^{sym}(A)$.

\begin{lemma}\label{L:sym-cyclic}
For every natural number $r>1$ and any set $A$, 
\[\chi\left(\Cyc_r^{sym}(A)\right)=
\begin{cases}
2 & \text{$r$ is even}\\
3 & \text{$r$ is odd}.
\end{cases}\]
\end{lemma}
\begin{proof}
The graphs $\Cyc_r^{sym}(A)$ partitions into connected components, each of them a cycle graph on $r$ vertices. It is well known and easy to see that you need $2$ colors to color even cycle graphs and $3$ colors to color odd cycle graphs.
\end{proof}
\end{example}

The next two examples are somewhat similar and they both have very small chromatic number. We define them and prove that their chromatic number is $2$.

\begin{example}[Denumerable Tuples Symmetric Shift Graph]\label{E:dnm-shft}
Let $A$ be an infinite set. The \emph{denumerable tuples symmetric shift graph} $\Sh_\omega^{sym}(A)$ is defined similarly as the symmetric shift graph but with vertices $A^{\U \omega}$. There is an edge between two vertices $f$ and $g$ if $f(n)=g(n+1)$ for all $n<\omega$ (or vice-versa).
\end{example}

\begin{example}[Glued Increasing Symmetric Shift Graphs]\label{E:Glued-Inc-shft}
Let $\bar n=\langle n_i:i<\omega\rangle$ be a strictly increasing sequence of natural numbers.  We define the graph $\Sh^{sym}_{\bar n,u}(A)$, for an infinite set $A$. The vertices are injective functions $\coprod_{i<\omega} [0,n_i]\to A$. Thus every vertex can be written as $f=\coprod_{i<\omega} f_i$.  We will say that there is an edge between two vertices $f$ and $g$ if \[f_i(m)=g_i(m+1)\] for every $0\leq m< n_i$, $i<\omega$ or
\[g_i(m)=f_i(m+1)\] for every $0\leq m< n_i$, $i<\omega$.
\end{example}

\begin{lemma}\label{L:Z-action}
Let $X$ be a set. Suppose $(n,x)\mapsto n+x$ is a free action of $\mathbb{Z}$ on $X$. We define a graph relation $\E$ on $X$ by setting that $x\E y$ if either $1+x=y$ or $1+y=x$. Then $\chi(X,E)=2$.
\end{lemma}
\begin{proof}
Since the action is free, for any $x,y\in X$ in the same $\mathbb{Z}$-orbit, $D_{x,y}=z$ for any $z\in\mathbb{Z}$ satisfying $z+x=y$ is well defined.

Let $\{x_r:r\in X/\mathbb{Z}\}$ be a set of representatives of the different orbits of the $\mathbb{Z}$ action on $X$.
We define a coloring $c:(X,E)\to \{0,1\}$ in the following way: For any vertex $x$ let $c(x)=\text{parity}(D_{x,x_r})$, where $x_r$ is the representative of the $\mathbb{Z}$-orbit of $x$.

We need to show that it is a legal coloring. Assume that $x$ and $y$ are connected by an edge. This means, without loss of generality, that $x=1+y$. But now by definition we cannot have $c(x)=c(y)$. Indeed, since $x$ and $y$ are in the same $\mathbb{Z}$-orbit, there are some $n$ and $m$ such that $x=n+x_r$ and $y=m+x_r$, for $r=\mathbb{Z}+x$. Consequently, $x=(n-m)+y$, and by freeness $n-m=1$. As a result, $c(x)\equiv c(y)+1 (\text{mod }2)$.
 
\end{proof}

\begin{lemma}\label{L:chr-dnm+gld-incr-shft}
For any infinite set $A$ and a strictly increasing sequence of natural numbers $\bar n$, $\chi\left(\Sh^{sym}_\omega(A)\right)=\chi\left( Sh_{\bar n,u}^{sym}(A)\right)=2$.
\end{lemma}
\begin{proof}
The proof for these two graphs are the same, albeit the definitions are slightly different. We prove for $\Sh_\omega^{sym}(A)$ and present the appropriate definitions for $\Sh_{\bar n,u}^{sym}(A)$ at the end.

Let $X\subseteq A^\omega$ be the set of all functions $f$ which are eventually injective, i.e. there exists an $n$ such that $f\restriction [n,\infty)$ is injective.

Fix some element $e\in A$. The integers $\mathbb{Z}$ acts on $X$ by translation:
if $z\in \mathbb{Z}$ and $f\in X$ then we define 
\[(z+f)(m)=
\begin{cases}
f(m-z) & 0\leq m-z\\
e & \text{otherwise}
\end{cases}\]

We define an equivalence relation $R$ on $X$:
\[f\R g \iff \exists n (f\restriction [n,\infty)=g\restriction [n,\infty)).\]
Note that if $f\R g$ and $z\in \mathbb{Z}$ then $z+f\R z+g$, so the $\mathbb{Z}$-action induces an action on $X/R$. We note that if $z+[f]=[f]$  for $f\in X$ (and $[f]$ being the class of $f$ in $X/R$) then $z=0$ by eventual injectivity of $f$. Or in other words, the $\mathbb{Z}$-action on $X/R$ is free.

Since if $f,g\in \Sh_\omega^{sym}(A)$ are connected by an edge then either $[f]=1+[g]$ or $[g]=1+[f]$, by Lemma \ref{L:Z-action} and Lemma \ref{L:basic-prop-chi}(3), $\chi\left( \Sh_\omega^{sym}(A)\right)=2$.
%
%

For $\Sh_{\bar n,u}^{sym}(A)$ we define:

Let $X$ be the set of all functions $f:\coprod_{i<\omega}\to [0,n_i]$ satisfying the property that there exists an $n$ such that for all $i<\omega$, $f_i\restriction [n,n_i-n]$ is injective.

For every $z\in \mathbb{Z}$ and $f\in X$ we define for $i<\omega$
\[(z+f)_i(m)=
\begin{cases}
f_i(m-z) & 0\leq m-z\leq n_i \\
e & \text{otherwise}
\end{cases}
\]

We define an equivalence relation $\R$ on $X$:
\[f\R g \iff \exists n \forall i<\omega (f_i\restriction [n,n_i-n]=g_i\restriction [n,n_i-n]).\]
\end{proof}
%
On the other hand if the glued shift graphs are bounded the picture is different.

\begin{example}[A Sequence of Bounded Shift Graphs]
Let $n$ be a natural number, $I$ a set and $\bar n=\langle n_i :i\in I\rangle$ a sequence of natural numbers satisfying $0<n_i\leq n$ for all $i\in I$. We define $\Sh_{\bar n,b}(A)$ for an infinite linearly ordered set $(A,<)$. The vertices are sequences of functions $f=(f_i)_{i\in I}$, where each $f_i:[0,n_i]\to A$ is order preserving. We will say that there is an edge between two vertices $f$ and $g$ if $f_i(m)=g_i(m+1)$ for every $0\leq m<n_i$, $i\in I$ (or vice-versa). 

\begin{lemma}\label{L:embed-symmetric-shift}
Let $(A,<)$ be an infinite linearly ordered set, $\bar n=\langle n_i:i\in I\rangle$ a uniformly bounded sequence of natural numbers with $n_i\geq 1$ and let $n=\max_{i \in I}\{n_i\}$. Then there exists an injective homomorphism  $\Sh_{n+1}(A)\to \Sh_{\bar n,b}(A)$.
\end{lemma}
\begin{proof}
For any tuple $u\in (A^{\U{n+1}})_{<}$ we define a vertex $f_u\in \Sh_{\bar n,b}(A)$. For every $0\leq h \leq n_i$, $i \in I$, we set $(f_u)_i(h)=u(h)$. Set $f=(f_i)_{i\in I}$. Note that if $0\leq h<h^\prime\leq n_i$ then $u(h)<u(h^\prime)$ so $(f_u)_i(h)< (f_u)_i(h^\prime)$. By the choice of $n$, $u\mapsto f_u$ is injective as well. 
%

We show that $u\mapsto f_u$ is a homomorphism. Assume that, without loss of generality, $u(h)=v(h+1)$ for every $0\leq h<n$. For every $i\in I$ and for every $0\leq h<n_i$ 
\[(f_u)_i(h)=u(h)=v(h+1)=(f_v)_i(h+1).\]
As needed. 
\end{proof}
\end{example}

The following propositions will be the backbone behind the main results.

\begin{proposition}\label{P:finding-a-shift-graph-no-order}
Let $A$ be an infinite set, $\lambda$ a cardinal with $2^\lambda\leq |A|$  and $G=(A^{\U \lambda},E)$ a graph on $A^{\U \lambda}$. If $\kap$ is an infinite regular cardinal satisfying
\begin{enumerate}
\item[(1)] $\beth_2(\lambda)<\kap$,
\item[(2)] $\chi(G)\geq\kap$ and
\item[(3)] for all $\bar a,\bar b, \bar c,\bar d\in A^{\U \lambda}$ if $\bar a\mathrel{E} \bar b$ and $f_{\bar a,\bar b}=f_{\bar c,\bar d}$ then $\bar c \mathrel{E} \bar d$
\end{enumerate}
then there exists an $n\in \mathbb{N}$ and an injective homomorphsim from $\Sh_n(\omega)$ to $G$.

\end{proposition}
\begin{proof}
Let $F=\{f_{\bar a,\bar b}:\bar a \mathrel{E}\bar b\}$ be the collection of all functions arising as $f_{\bar a,\bar b}$ for some $\bar a$ and $\bar b$ sharing an edge. If we set $E_f=\{(\bar a,\bar b ):f=f_{\bar a,\bar b}\vee f=f_{\bar b,\bar a}\}$ 
then, since by assumption $(3)$, $E=\bigcup_{f\in F} E_f$, then by Lemma \ref{L:basic-prop-chi}(2) 
\[\kap\leq \prod_{f\in F}\chi\left( V,E_f\right).\]
Since $|F|\leq 2^\lambda$ we may assume there exists $f\in F$ with $\chi(V,E_f)>2^\lambda\geq \aleph_0$. Replace $G$ by $(V,E_f)$. Note that although now we only have that $\chi(V,E)\geq \aleph_0$, we gained that $\bar a$ and $\bar b$ are connected by an edge if and only if $f_{\bar a,\bar b}=f$ or $f_{\bar b,\bar a}=f$.

For any $\beta\in\dom(f)\subseteq \lambda$, we distinguish between four possibilities:
\begin{enumerate}
\item ``$\beta$ is a fixed point": $f(\beta)=\beta$;
\item ``$\beta$ generates a finite cycle": there exists a natural number $1<n\in \mathbb{N}$ such that $f^n(\beta)=\beta$ and $f^{n-1}(\beta)\neq \beta$;
\item ``$\beta$ generates a finite shift": there exists a natural number $0<n\in \mathbb{N}$ such that $f^n(\beta)\notin \dom (f)$;
\item ``$\beta$ generates an infinite shift": the set $\{f^n(\beta):n<\omega\}$ is infinite.
\end{enumerate}

We first note the following observations, which will allow us to cross out some of the possibilities:

\begin{itemize}
\item \underline{No $\beta\in \dom (f)$ generates a finite cycle.} Assume there exists $\beta\in \dom (f)$ and $1<n<\omega$ such that $f^{n}(\beta)=\beta$ and $f^{n-1}(\beta)\neq \beta$. Define a homomorphism $G\to \Cyc_n^{sym}(A)$ which maps $\bar a$ to $(a_\beta,a_{f(\beta)},\dots a_{f^{n-1}(\beta)})$. It is a homomorphism because if there is an edge between $\bar a$ and $\bar b$ then by definition of $f$, $a_\beta=b_{f(\beta)},\dots,a_{f^{n-1}(\beta)}=b_\beta$. By Lemma \ref{L:sym-cyclic} and Lemma \ref{L:basic-prop-chi}(3), $\chi(G)\leq \chi(\Cyc_n^{sym}(A))\leq 3$, contradiction.

\item \underline{No $\beta\in \dom(f)$ generates an infinite shift.} Assume there exists $\beta<\lambda$ with $\{f^{n}(\beta): n<\omega\}$ infinite. Define a homomorphism $G\to \Sh_{\omega}^{sym}(A)$ which maps $\bar a$ to $n\mapsto a_{f^n(\beta)}$. By definition this is a homomorphism of graphs. By Lemma \ref{L:chr-dnm+gld-incr-shft} and Lemma \ref{L:basic-prop-chi}(3), $\chi(G)\leq \chi(\Sh_{\omega}^{sym}(A))=2$, contradiction.
\end{itemize}

Let $I=\dom(f)\setminus \Rg(f)$. For any $\beta\in I$ let $n_\beta$ be the maximal natural number $n\geq 1$ such that $f^{n-1}(\beta)\in \dom (f)$.  Note that \[\dom (f)=\bigcup_{\beta \in I} \{\beta,\dots, f^{n_\beta-1}(\beta)\}\cup \{\beta<\lambda: f(\beta)=\beta\}\]
and that 
 \[\dom (f)\cup\Rg(f)=\bigcup_{\beta \in I} \{\beta,\dots, f^{n_\beta}(\beta)\}\cup \{\beta<\lambda: f(\beta)=\beta\}.\]

\begin{claim}
There exists a uniform bound on $\{n_\beta:\beta\in I\}$.
\end{claim}
\begin{claimproof}
Otherwise, assume there are $\langle \beta_{i}:i<\omega\rangle$ such that the sequence $\bar n:=\langle n_{\beta_i}: i<\omega\rangle$ is strictly increasing. 
We define a homomorphism from $G$ to $\Sh_{\bar n,u}^{sym}(A)$ similarly as before. Consequently, $\chi(G)\leq 2$ (by using Lemma \ref{L:chr-dnm+gld-incr-shft}, Lemma \ref{L:basic-prop-chi}(3) and the relevant homomorphisms), contradiction.
\end{claimproof}

We are thus left with two cases:

\underline{Case 1:} $I=\emptyset$. Thus $f$ is the identity on $\dom (f)$. If $\dom (f)=\lambda$ then $G$ is an anticlique and can thus can be colored by only one color, contradiction. Hence $\dom (f)\subsetneq \lambda$. Let $\bar a\in G$ be any sequence. Consider the induced subgraph $G_0\subseteq G$ with vertices:
\[\{\bar b\in G: (\forall i\in\dom (f))\,(b_{i}=a_{i})\}.\]   
By the definition of the edge relation $G_0$ is a complete graph of size $|A|$. In particular we may embed the complete graph on $\omega$ as a subgraph.

\underline{Case 2:} $I\neq\emptyset$. Let $n=\max_{\beta\in I}\{n_\beta\}$ and let $\phi:\lambda\times \omega\times (\Sh_{\bar n,b}(\omega)\cup\{0,1\}) \to A$ be an injective function, which exists since $(\aleph_0)^\lambda+\aleph_0+\lambda\leq |A|$.

We define an injective homomorphism from $\Sh_{\bar n,b}(\omega)$ into $G$ where $\bar n=\langle n_\beta :\beta \in I\rangle$. 
For every function $\mu=(\mu_\beta)_{\beta \in I}$, where $\mu_\beta:[0,n_\beta]\to \omega$ is order preserving, we associate an injective function $\psi_\mu:\lambda \to A$ as follows. For every $\beta\in I$ and $h\in [0,n_\beta]$ we define
\[\psi_\mu(f^h(\beta))=\phi(\beta,\mu_\beta(h),0),\]
note that this is well defined. For every $\alpha\not\in \bigcup_{\beta \in I} \{\beta,\dots, f^{n_\beta}(\beta)\}$ such that $f(\alpha)=\alpha$ we define
\[\psi_\mu(\alpha)=\phi(\alpha,0,1)\]
and otherwise we define
\[\psi_\mu(\alpha)=\phi(\alpha,0,\mu).\]
We claim that the map $\mu\mapsto \psi_\mu$ is an injective homomorphism. 

Injectivity: Let $\mu,\nu\in \Sh_{\bar n,b}(\omega)$ with $\psi_\mu=\psi_\nu$. Let $\beta\in I$ and $h\in [0,n_\beta]$. Since $\psi_\mu(f^h(\beta))=\psi_\nu(f^h(\beta))$ and $\phi$ is injective, $\mu_\beta(h)=\nu_\beta(h)$.

Homomorphism: Assume that $\mu$ and $\nu$ are connected by an edge, i.e. without loss of generality for every $\beta\in I$ and $h\in [0,n_\beta)$, $\mu_\beta(h)=\nu_\beta(h+1)$. We need to show that for every $i,j<\lambda$, $\psi_\mu(i)=\psi_\nu(j)$ if and only $f(i)=j$.

Assume that $f(i)=j$. In particular, $i\in \dom(f)$. If $i=\beta=f(\beta)=j$ then $\psi_\mu(\beta)=\phi(\beta,0,1)=\psi_\nu(\beta)$. If $i\in \dom(f)\setminus \Rg(f)$ then $i=f^h(\beta)$ for some $\beta\in I$ and $h\in [0,n_\beta)$. Thus 
\[\psi_\mu(i)=\psi_\mu(f^h(\beta))=\phi(\beta,\mu_\beta(h),0)=\]\[\phi(\beta,\nu_\beta(h+1),0)=\psi_\nu(f^{h+1}(\beta))=\psi_\nu(j).\]

Assume that $\psi_\mu(i)=\psi_\nu(j)=e$. Since $\mu\neq \nu$, by the injectivity of $\phi$ we have only two possibilities: either $e=\phi(\cdots, 1)$ or $e=\phi(\cdots,0)$. If the former happens, necessarily $f(i)=i$, $f(j)=j$ and $i=j$. 

Otherwise, $i=f^h(\beta)$ and $j=f^{h^\prime}(\beta^\prime)$ for some $\beta,\beta^\prime\in I$, $h\in [0,n_\beta]$ and $h^\prime\in [0,n_{\beta}^\prime]$. Also, since
\[\phi(\beta,\mu_\beta(h),0)=\phi(\beta^\prime,\nu_{\beta^\prime}(h^\prime),0),\]
$\beta=\beta^\prime$ and $\mu_\beta(h)=\nu_{\beta}(h^\prime)$. If $h\in [0,n_\beta)$ then $\mu_\beta(h)=\nu_\beta(h+1)$ since $\mu$ and $\nu$ are connected by an edge, so since $\nu_\beta$ is injective $h^\prime=h+1$. So $f(i)=j$. Otherwise, $h=n_\beta$. If $h^\prime>0$ then since $\nu_\beta(h^\prime)=\mu_\beta(h^\prime-1)=\mu_\beta(h)$ we get a contradiction to the injectivity of $\mu_\beta$. Consequently it must be that $h^\prime=0$. This contradicts the fact that $\nu_\beta$ is order preserving and $n_\beta\geq 1$.

Applying Lemma \ref{L:embed-symmetric-shift} we may conclude that there exists an injective homomorphism from $\Sh_{n+1}(\omega)$ into $\Sh_{\bar n,b}(\omega)$, and thus into $G$ as well.
\end{proof}

\begin{proposition}\label{P:finding-a-shift-graph-ordered}
Let $(A,<)$ be an infinite linearly ordered set, $m<\omega$ and $G=((A^{\U m})_<,E)$ a graph on $(A^{\U m})_<$. Assume $\chi(G)\geq\aleph_0$ and that for all $\bar a,\bar b, \bar c,\bar d\in (A^{\U m})_<$ if $\bar a\mathrel{E} \bar b$ and $f_{\bar a,\bar b}=f_{\bar c,\bar d}$ then $\bar c \mathrel{E} \bar d$. Then there exists $n<\omega$ such that $G$ contains all finite subgraphs of $\Sh_n(\omega)$.

\end{proposition}
\begin{proof}
As was done in the proof of Proposition \ref{P:finding-a-shift-graph-no-order}, letting $F=\{f_{\bar a,\bar b}:\bar a \mathrel{E}\bar b\}$, since $|F|<\aleph_0$, we may assume that $E=E_f$ for some $f\in F$ (see Definition \ref{D:E_f,D_f}).

Since the tuples are increasing, $f$ is necessarily an order preserving function ($i<j\in \dom(f) \implies f(i)<f(j)$). Thus, as $m$ is finite, for any $\beta\in\dom(f)\subseteq m$ with $f(\beta)\neq \beta$, ``$\beta$ generates a finite shift" (in the context of Proposition \ref{P:finding-a-shift-graph-no-order}), i.e. there exists a natural number $0<n\in \mathbb{N}$ such that $f^n(\beta)\notin \dom (f)$.

Let $I=\dom(f)\setminus \Rg(f)$. For any $\beta\in I$ let $n_\beta$ be the maximal natural number $n\geq 1$ such that $f^{n-1}(\beta)\in \dom (f)$. Note that \[\dom (f)=\bigcup_{\beta \in I} \{\beta,\dots, f^{n_\beta-1}(\beta)\}\cup \{\beta<m: f(\beta)=\beta\}\]
and that 
 \[\dom (f)\cup\Rg(f)=\bigcup_{\beta \in I} \{\beta,\dots, f^{n_\beta}(\beta)\}\cup \{\beta<m: f(\beta)=\beta\}.\]

As in the proof of Proposition \ref{P:finding-a-shift-graph-no-order}(Case 1), we may assume that $I\neq \emptyset$. Say that $\beta\in I$ is \emph{increasing} if $\beta< f(\beta)$ and \emph{decreasing} otherwise (equivalently, $f(\beta)<\beta$). Also, as $f$ is order preserving and the tuples are increasing, we may find a partition $m=J_1\cup\dots\cup J_N$ satisfying that
\begin{list}{•}{}
\item each of the $J_i$ are convex and $J_1<\dots<J_N$;
\item if $\beta\in J_i\cap \dom(f)$ then $f(\beta)\in J_i$;
\item if $\beta\in I\cap J_i$ is increasing then every $\beta^\prime\in I\cap J_i$  is increasing;
\item if $\beta\in I\cap J_i$ is decreasing then every $\beta^\prime\in I\cap J_i$  is decreasing and
\item if for $\beta\in J_i$, $f(\beta)=\beta$ then for every $\beta^\prime\in (\dom(f)\cup\Rg(f))\cap J_i$, $\beta=\beta^\prime$.
\end{list}

For every $1\leq i\leq N$, set $f_i=f\cap (J_i\times J_i)$. 

For any $1\leq i\leq N$ if $J_i$ is of increasing type, by applying Lemma \ref{L:embedding-intertwined} there is a homomorphism $g_{i,k}:\LSh_{k}(\omega)\to ((\mathbb{Q}^{\U{J_i}})_<,D_{f_i})$ for any large enough $k$.

For any $1\leq i\leq N$ if $J_i$ is of decreasing type, by applying Lemma \ref{L:embedding-intertwined} to $(\mathbb{Q},<^*)$ (there reverse order on $\mathbb{Q}$) and $(J_i,<^*)$ (the reverse order on $J_i$) there is a homomorphism $g_{i,k}^*:\LSh_{k}(\omega)\to ((\mathbb{Q}^{\U{{(J_i,<^*)}}})_{<^*},D_{f_i})$ for any large enough $k$. Since the identity function is an isomorphism of directed graphs 
\[((\mathbb{Q}^{\U{{(J_i,<^*)}}})_{<^*},D_{f_i})\cong ((\mathbb{Q}^{\U{J_i}})_{<},D_{f_i}),\]
we may compose an get a homomorphism $g_{i,k}:\LSh_{k}(\omega)\to ((\mathbb{Q}^{\U{J_i}})_{<},D_{f_i})$.

Let $k$ be large enough so that $g_{i,k}$ are defined for all $i$ and set $g_i=g_{i,k}$.

For any $1\leq i\leq N$, if $J_i$ is of constant type fix some embedding $g_i:(J_i,<)\to (\mathbb{Q},<)$.

Let $(A,<)\prec (\mathcal{A},<)$ be a sufficiently saturated extension with $(\mathcal{A},<)$ containing $(Q,<)=(N\times \mathbb{Q}\times (\LSh_k(\omega)\cup\{0\}),<)$ as a substructure, where we may chose any linear order on $\LSh_k(\omega)\cup\{0\}$.
Since the inclusion $(Q,<)\subseteq (\mathcal{A},<)$ induces an injective homomorphism \[((Q^{\U m})_<, D_f)\to ((\mathcal{A}^{\U m})_<, D_f),\] we may assume that $\mathcal{A}=Q$. We will construct a homomorphism $\LSh_k(\omega)\to ((\mathcal{A}^{\U m})_<,D_f)$.

Let $\mu\in \LSh_k(\omega)$. We define $\psi_\mu\in (\mathcal{A}^{\U m})_<$ as follows. If $\alpha\in J_i$, with $J_i$ increasing or decreasing then
\[\psi_\mu(\alpha)=(i,g_i(\mu)(\alpha),0).\]
If $\alpha\in J_i$, with $J_i$ of constant type, and $f(\alpha)=\alpha$ then
\[\psi_\mu(\alpha)=(i,g_i(\alpha),0).\]
If $\alpha\in J_i$, with $J_i$ of constant type, and $f(\alpha)\neq \alpha$ then
\[\psi_\mu(\alpha)=(i,g_i(\alpha),\mu).\]
Since $J_1<\dots<J_N$ then by definition, $\psi_\mu$ is increasing. We claim that $\mu\mapsto \psi_\mu$ is a homomorphism.

Assume that $\mu,\nu\in \LSh_k(\omega)$ are such that $\mu(h)=\nu(h+1)$ for $0\leq h<k-1$. We will show that $f(\alpha)=\beta$ if and only if $\psi_\mu(\alpha)=\psi_\nu(\beta)$.

If $f(\alpha)=\beta$ then $\alpha,\beta\in J_i$ for some $1\leq i\leq N$. If $J_i$ is not of constant type then since $g_i$ is a homomorphism, $g_i(\mu)(\alpha)=g_i(\nu)(\beta)$, so
\[\psi_\mu(\alpha)=(i,g_i(\mu)(\alpha),0)=(i,g_i(\nu)(\beta),0)=\psi_\nu(\beta).\]
If $J_i$ is of constant type then $\beta=f(\alpha)=\alpha$ and 
\[\psi_\mu(\alpha)=(i,g_i(\alpha),0)=(i,g_i(\beta),0)=\psi_\nu(\beta).\]

Now assume that $\psi_\mu(\alpha)=\psi_\nu(\beta)$. By definition, $\alpha,\beta\in J_i$ for some $1\leq i\leq N$. If $J_i$ is of constant type then since $\mu\neq \nu$ by the definition of $\psi$, $f(\alpha)=\alpha$, $f(\beta)=\beta$ and $g_i(\alpha)=g_i(\beta)$. Consequently, $\alpha=\beta$ and as a result $\alpha=\beta=f(\alpha)$.

If $J_i$ is not of constant type then $g_i(\mu)(\alpha)=g_i(\nu)(\beta)$. By the fact that $g_i$ is a homomorphism, $f(\alpha)=\beta$.

Finally, by Proposition \ref{P:homomorphism-is-enough}, $G$ contains all finite subgraphs of $\Sh_k(\omega)$.

\end{proof}

\section{Superstable and  \texorpdfstring{$\omega\text{-stable Graphs}$}g}
We use the main result of the previous section in order to prove the strong from of Taylor's conjecture for $\omega$-graphs and a suitable variant for superstable graphs.

The following result is somewhat reminiscent (in flavor) of \cite[Theorem III]{ER-comb}. It is a local version of the the well known fact that, in stable theories, every indiscernible sequence is an indiscernible set \cite[Lemma 9.1.1]{TZ} (it is possibly known, but could not find a reference).

Generalizing the notation from Definition \ref{D:f_a,b}, for two tuples, possibly of different length, $\bar a$ and $\bar b$, we denote $f_{\bar a,\bar b}=\{(i,j):a_i=b_j\}$.

Recall that for a set of formulas $\Delta$, a \emph{$\Delta$-indiscernible sequence} is a sequence of elements that are indiscernible only with respect to formulas from $\Delta$. For a formula $\varphi(x_0,\dots, x_{n-1})$ let $\Delta_\varphi:=\{\varphi(x_{\pi(0)},\dots,x_{\pi(n-1)}): \pi \text{ is a function fron $n$ to $n$} \}$.

\begin{proposition}\label{P:stable-edge-relation}
Let $T$ be a complete theory and $\varphi(x,y)$ a partitioned stable formula, with $x$ and $y$ possibly of different lengths. Let $I$ be a $\Delta_\varphi$-indiscernible sequence indexed by an infinite linearly ordered set $(Q,<)$.

If there exist $\bar a\in I^{\U{|x|}}$ and $\bar b\in I^{\U{|y|}}$ disjoint increasing tuples such that $\varphi(\bar a,\bar b)$ holds then for every disjoint increasing tuples $\bar c\in I^{\U{|x|}}$ and $\bar d\in I^{\U{|y|}}$, $\varphi(\bar c,\bar d)$ holds.

Moreover, if $\bar a,\bar c\in I^{\U{|x|}}$ and $\bar b,\bar d\in I^{\U{|y|}}$ are increasing tuples then
\begin{itemize}
\item[($\star$)] $\varphi(\bar a,\bar b)\wedge f_{\bar a,\bar b}=f_{\bar c,\bar d}\implies \varphi(\bar c,\bar d).$
\end{itemize}
\end{proposition}
\begin{proof}
Let $I^\prime$ be an indiscernible sequence with the same EM-type as $I$. Suppose ($\star$) is not true as witnessed by $\bar a,\bar b,\bar c,\bar d$, then let $\bar a^\prime, \bar b^\prime, \bar c^\prime, \bar d^\prime$ in $I^\prime$ be such that $\bar a^\prime\bar b^\prime$ as the same order type as $\bar a\bar b$, and $\bar c^\prime\bar d^\prime$ has the same order type as $\bar c\bar d$. It follows that ($\star$) is not true for $\bar a^\prime,\bar b^\prime,\bar c^\prime, \bar d^\prime$. We may thus assume that $I$ is an indiscernible sequence. Similarly, we may assume that $(Q,<)$ is a dense linear order with no endpoints. Also, we endow $I$ with the order induced by $Q$, i.e. we write $a_i<a_j$ but mean $i<j\in Q$.

Let $\bar a,\bar b,\bar c$ and $\bar d$ be as in the statement. By applying an automorphism we may assume that $\bar c=\bar a$. Let $X=\{\bar c\in I\setminus \bar a: \varphi(\bar a,\bar c)\}$. Note $\bar b\in X$.

By stability the $\varphi$-type $\tp_\varphi (\bar a/I\setminus \bar a)$ is definable, i.e. there is some formula $\psi(y,\bar e)$ with $\bar e\in I\setminus \bar a$ such that 
\[ \bar c\in X \iff \bar c\models\psi(y,\bar e).\]

Let $\bar h\in I\setminus \bar a\bar e$ have the same order type as $\bar e$ over $\bar a$, which exists by density. Let $\sigma$ be an automorphism of $(I,<)$ which fixes $\bar a$ and maps $\bar e$ to $\bar h$. By indiscernibility it follows that $\sigma(X)=X$.

We claim that $X$ is definable over both $\bar e$ and $\bar h$ in the structure $(I\setminus \bar a,<)$. Indeed if $\bar c_1 ,\bar c_2$ have the same order type over $\bar e$ then $\bar c_1\in X\iff \bar c_2\in X$. Since there only finitely many order types over $\bar e$, this shows the claim for $\bar e$. As $\sigma(X)=X$, we have it also for $\bar h$.


%
As DLO eliminates imaginaries \cite[Exercise 8.4.3]{TZ}, $X$ has a code $\ulcorner X\urcorner\in \dcl(\bar e)\cap \dcl(\bar h)$. As $\dcl$ is trivial in DLO and $\bar e$ and $\bar h$ are disjoint, $X$ is definable over $\emptyset$. Consequently, either $X$ is empty or equals to $I\setminus \bar a$. Since $X$ is not empty, $X=I\setminus \bar a$. This proves the first part.

Now for the moreover part. Assume that $\varphi(\bar a,\bar b)$ holds and $f_{\bar a,\bar b}=f_{\bar c,\bar d}$. By applying an automorphism, we may assume that $\bar c=\bar a$. We prove it by induction on $|\dom (f)|$, where $f=f_{\bar a,\bar b}=f_{\bar a,\bar d}$, the case $|\dom(f)|=0$ being the first part of the proposition. 

Assume this is true for functions whose domains have cardinality less than $k$ and that $|\dom(f)|=k$. Let $i$ be the maximal element of $\dom(f)$.

Note that $I_{<a_{i}}$ is indiscernible over $I_{\geq a_{i}}$. Consider the formula $\psi(u,v)=\varphi(ua_{\geq i},vb_{\geq f(i)})$ (recall $b_{f(i)}=a_{i}$). Applying the induction hypothesis to $\psi(u,v)$ we conclude that $\varphi(\bar a,d_{<f(i)}b_{\geq f(i)})$ and so also $\varphi(\bar a,d_{\leq f(i)}b_{> f(i)})$ (because $d_{f(i)}=a_i=b_{f(i)}$). Now note that $I_{>a_i}$ is also indiscernible over $I_{\leq a_{i}}$ and we consider the formula $\psi(u,v)=\varphi(a_{\leq i}u,d_{\leq f(i)}v)$ (recall that $d_{f(i)}=a_{i}$). Applying the base of the induction hypothesis, we conclude that $\varphi(\bar a,\bar d)$, as required. 

%
%
%
%

\end{proof}
%
%
%

\begin{definition}
Let $\mathcal{L}$ be a first order language and $T$ a complete $\mathcal{L}$-theory with infinite models and let $\Delta$ be a set of formulas. An \emph{EM$_\Delta$-Model} of $T$ is a model which is generated by a $\Delta$-indiscernible sequence, i.e. a model $M\models T$ with a $\Delta$-indiscernible sequence $I$ such that for every $b\in M$ there exist a term $t(x_0,\dots,x_{n-1})$ and elements $a_0<\dots<a_{n-1}\in I$ with $b=t(a_0,\dots,a_{n-1})$. If $\Delta$ is the set of all formulas we omit $\Delta$ from the notation.
\end{definition}

For a binary relation $E$, let $\Delta (E)$ be the collection of formulas of the form $E(t(x),t(y))$, where $t$ is a term.

\begin{theorem}\label{T:EM-model}
Let $\mathcal{L}=\{E,\dots\}$ be a first order language with $E$ a binary relation. Let $T$ an $\mathcal{L}$-theory specifying that $E$ is a symmetric and irreflexive stable relation. Let $\kap>|T|+\aleph_0$ be an infinite regular cardinal. Let $G=(V;E,\dots)\models T$ be an EM$_{\Delta (E)}$-model. If $\chi(V,E)\geq \kap$ then there exists a natural number $n$ such that $G$ contains all finite subgraphs of $\Sh_n(\omega)$. 
\end{theorem}
\begin{proof}
Let $(A,<)$ be a linearly ordered set, $I=\langle r_i:i\in A\rangle$ an indiscernible sequence and $\{t_\alpha\}_{\alpha <|T|}$ a set of terms satisfying that $V=\bigcup_{\alpha<|T|} t_{\alpha}(I)$, where $t_{\alpha}(I)$ is the image of the map $I\mapsto V$ given by substituting increasing tuples in $t_\alpha$.


By Lemma \ref{L:basic-prop-chi}(1), $\kap\leq \chi(G)\leq \sum_{\alpha<|T|} \chi(t_\alpha(I),E\restriction t_\alpha(I))$. Since $\kap$ is a regular cardinal, there exists an $\alpha$ such that $\chi(t_\alpha(I),E\restriction t_\alpha(I))\geq \kap$. We may thus assume that $V=t(I)$ for some term $t=t(\bar x)=t(x_0,\dots,x_{n-1})$.

The map $t:(I^{\U n})_<\to t(I)$ induces a graph on $(I^{\U n})_<$  by specifying that $\bar a \E\bar b$ if and only if $t(\bar a)\E t(\bar b)$. By Lemma \ref{L:basic-prop-chi}(4), $\chi((I^{\U n})_<)\geq \kap$ as well.
%
%

Since the edge relation $E(v,u)$ is stable, so is $E(t(\bar x),t(\bar y))$. As a result, the moreover part of Proposition \ref{P:stable-edge-relation} allows us to apply Proposition \ref{P:finding-a-shift-graph-ordered}. Hence $(I^{\U n})_<$ contains all finite subgraphs of $\Sh_k(\omega)$ for some $k$. We may now conclude by applying the consequently part of Proposition \ref{P:homomorphism-is-enough}.
\end{proof}

\begin{corollary}\label{C:superstable,omegastable}
Let $G=(V,E)$ be a graph.  If 
\begin{itemize}
\item $G$ is superstable and  $\chi(G)> 2^{\aleph_0}$ or
\item $G$ is $\omega$-stable and $\chi(G)>\aleph_0$ 
\end{itemize}
then $G$ contains all finite subgraphs of $\Sh_n(\omega)$ for some $n\in \mathbb{N}$.
\end{corollary}
\begin{proof}
Suppose $G$ is superstable and $\chi(G)>2^{\aleph_0}$. By \cite[Claim 16.2(2B.c)]{Sh:1151} or \cite[page 345]{mariou}, \cite[Theorem 3.B]{mariouthesis} there exists $\{E\}\subseteq \mathcal{L}$ of cardinality $2^{\aleph_0}$ and an $\mathcal{L}$-saturated EM-model $\mathcal{G}$ such that $\mathrm{Th}(\mathcal{G})\restriction \{E\}=\mathrm{Th}(G)$. Since $\mathcal{G}$ is saturated, we may embed $G$ as an induced subgraph of $\mathcal{G}$. Since $\chi(\mathcal{G}) > 2^{\aleph_0}$, by Theorem \ref{T:EM-model} all finite subgraphs of $\Sh_n(\omega)$ are contained in $\mathcal{G}$ for some $n\in\mathbb{N}$.  The result now follows since $G\prec \mathcal{G}$, as graphs.

For $\omega$-stable graphs we may use \cite[Theorem C]{mariou} to find an $\mathcal{L}$-saturated EM-model in a countable language.
\end{proof}

\section{Stationary Stable Graphs}
The crucial part of the proof of Theorem \ref{T:EM-model} was the existence of a saturated EM-model.
It is a natural question to ask whether the technique from the previous section can be generalized to any stable graph, i.e. is the following true:
\begin{quote}
There exists a cardinal $\kappa$ such that for every stable graph with $\chi(G)\geq \kappa$ there exists a saturated EM-model $\mathcal{G}$, in an expansion $\{E\}\subseteq \mathcal{L}$ with $|\mathcal{L}|<\kappa$, such that $G\prec \mathcal{G}\restriction \{E\}$.
\end{quote}

However, Mariou has shown in \cite[Theorem 3.A]{mariouthesis} that if a stable theory $T$ has a $\kappa^+$-saturated EM-model in an expansion $\mathcal{L}$ with $|\mathcal{L}|\leq \kappa$ then $T$ is superstable. As a result, such a result would imply superstability. For general stable graphs a different approach is needed.

A connected notion to that of EM-models is the that of representations of structures from \cite{919}. We will need a variation on the theme.

\begin{definition}[The Free Algebra]
Suppose $A$ is a pure set. Let $\mathcal{M}_{\mu,\kappa}(A)$ be the (non first order) structure whose vocabulary is $\mathcal{L}_{\mu,\kappa}=\{F_{\alpha,\beta}: \alpha<\mu, \beta<\kappa\}$, where each $F_{\alpha,\beta}$ is a $\beta$-ary function symbol for all $\alpha<\mu$ (note that we allow infinite arity). The universe of $\mathcal{M}_{\mu,\kappa}(A)$ is $\bigcup_{\gamma\in \mathrm{Ord}} \mathcal{M}_{\mu,\kappa,\gamma}(A).$
Where 
\begin{list}{•}{}
\item $\mathcal{M}_0(A)=A$,
\item for limit $\gamma$, $\mathcal{M}_\beta(A)=\bigcup_{\gamma\prime <\gamma}\mathcal{M}_{\gamma\prime}(A)$,
\item and for successor
\[\mathcal{M}_{\gamma+1}=\mathcal{M}_\gamma\cup\{F_{\alpha,\beta}(\bar b):\bar b\in (\mathcal{M}_\gamma)^\beta,\, \alpha<\mu,\, \beta<\kappa\}.\]
\end{list}
We treat $F_{\alpha,\beta}(\bar b)$ as a new formal object.
\end{definition}

For a cardinal $\kappa$, let $\reg (\kappa)$ be $\kappa^+$ if $\kappa$ is singular and $\kappa$ otherwise.

\begin{fact}\cite[Remark 2.3]{919}\label{F:cohenshelah}
Let $A$ and $\mathcal{M}_{\mu,\kappa}(A)$ be as before.
$\mathcal{M}_{\mu,\kappa}(A)$ is a set whose cardinality is at most $(|A|+\mu)^{<\reg (\kappa)}$ (though defined as a class).
%
\end{fact}

%

\begin{remark}\label{R:bound-terms-freealg}
Fixing a set of variables $X=\{x_i:i<\reg(\kappa)\}$, the set of terms in $\mathcal{L}_{\mu,\kappa}$ in $X$ can be identified with $\mathcal{M}_{\mu,\kappa}(X)$. It follows from Fact \ref{F:cohenshelah} that their number is bounded by $(\reg(\kappa)+\mu)^{<\reg(\kappa)}$.
\end{remark}

For any permutation $\pi$ of $A$ we denote by $\widehat \pi$ the induced automorphism of $\mathcal{M}_{\mu,\kappa}(A)$.

\begin{definition}\label{D:homog-repres}
Let $M$ be a structure. A \emph{homogeneous representation} of $M$ in $M_{\mu,\kappa}(A)$ is a function $\Phi:M\to \mathcal{M}_{\mu,\kappa}(A)$ satisfying 
\begin{enumerate}
\item For every term $t(\bar x)$, where $\bar x$ is tuple of length $\beta<\kappa$ containing the variables of $t$, if $t(\bar a)\in \Img(\Phi)$ for some $\bar a\in A^{\U \beta}$ then $t(\bar b)\in \Img(\Phi)$ for all $\bar b\in A^{\U \beta}$; 
\item For any two finite sequences $\bar a,\bar b \in M^{\U n}$, if there exists an permutation $\pi$ of $A$ such that $\widehat \pi (\Phi(a_i))=\Phi(b_i)$, for all $i<n$, then \[\tp^{M}(\bar a)=\tp^{M}(\bar b).\]
\end{enumerate}

We say that $\Phi$ is a \emph{skeletal homogeneous representation} if it is an injective partial function satisfying $(1)$ and $(2)$ on its domain and that $\dcl(\dom (\Phi))=M$. 
\end{definition}

\begin{remark}\label{R:original-rep}
Representations were originally defined in \cite[Definition 2.1]{919} and the  definition was that of a function $\Phi:M\to\mathcal{M}_{\mu,\kappa}(A)$ satisfying that 
\[\mathrm{qftp}(\Phi(\bar a))=\mathrm{qftp}(\Phi(\bar b))\implies \tp^M(\bar a)=\tp^M(\bar b).\]
Since every permutation of $A$ lifts to an automorphism of the free algebra, the antecedent in condition (2) implies that $\Phi(\bar a)$ and $\Phi(\bar b)$ have the same quantifier-free type. As a result, every representation satisfies condition (2) of a homogeneous representation.
\end{remark}

\begin{proposition}\label{P:embed-shift-represen}
Let $M$ be a structure and assume there exists a skeletal homogeneous representation $\Phi:\dom(\Phi)\to \mathcal{M}_{\mu,\kappa}(A)$ of $M$, where $A$ is a pure set, $\kappa$ and $\mu$ are infinite, and that
\begin{enumerate}
\item $(\reg(\kappa)+\mu)^{<\reg(\kappa)}<\kap$,
\item $\beth_2(\lambda)<\kap$ for all $\lambda<\reg(\kappa)$,
\item $2^{<\reg(\kappa)}\leq |A|$.
\end{enumerate}
For every graph $G=(V,E)$ that is $\emptyset$-interpretable in $M$ with $\chi(G)\geq \kap$, where $\kap$ is a regular cardinal, there exists an $n\in\mathbb{N}$ such that $G$ contains all finite subgraphs of $\Sh_n(\omega)$.
\end{proposition}
\begin{proof}
Since $G$ is interpretable in $M$ there exist $r\in \mathbb{N}$, a definable subset $V_0\subseteq M^r$ and an interpretation $g:V_0\to V$ (see \cite[Section 5.3]{hodges}). By definition, $g$ is surjective and $G_0=(V_0,g^{-1}(E))$ is a definable graph. Thus $g$ is a surjective homomorphism and by Lemma \ref{L:basic-prop-chi}(4) $\chi(G)=\chi(G_0)$. Note that if $G_0$ contains all finite subgraphs of $\Sh_n(\omega)$ then by Proposition \ref{P:homomorphism-is-enough} so does $G$ (maybe for a different $n$). Consequently, we may assume that the graph $G=(V,E)$ is $\emptyset$-definable in $M^r$.

Let $D=\dom (\Phi)$ and let $\langle f_i(\bar v_i): i<\omega\rangle$ be an enumeration of all $\emptyset$-definable functions to $M$. 
Let \[\mathcal{U}=\{F_{i,|\bar v_i|}(b_0,\dots, b_{|\bar v_i|-1}): b_0,\dots,b_{|\bar v_i|-1}\in \Img(\Phi),\, i<\omega\}\subseteq \mathcal{M}_{\mu,\kappa}(A).\]
Define a surjective map $\Psi_0:\mathcal{U}\to M$ by mapping $F_{i,|\bar x_i|}(b_0,\dots, b_{|\bar v_i|-1})$ to \[f_i(\Phi^{-1}(b_0),\dots,\Phi^{-1}(b_{|\bar v_i|-1})).\] Note that $\Phi$ is injective so this is well defined.

Let $\Psi_1=(\Psi_0)^r:(\mathcal{U})^r\to M^r$, $\mathcal{V}=\{a\in (\mathcal{U})^r: \Psi_1(a)\in V\}$ and $\Psi=\Psi_1\restriction \mathcal{V}:\mathcal{V}\to V$. Let $\mathcal{E}=\Psi^{-1}(E)$, hence $\mathcal{G}=(\mathcal{V},\mathcal{E})$ is a graph and note that $\chi(G)=\chi(\mathcal{G})$ by Lemma \ref{L:basic-prop-chi}(4). Let $\nu =(\reg(\kappa)+\mu)^{<\reg(\kappa)}$.

Let $X=\{x_i: i<\reg (\kappa)\}$ be a set of variables as in Remark \ref{R:bound-terms-freealg}. Let $\mathfrak{t}_0$ be the set of pairs $(t,\bar x)$, where $\bar x$ is a sequence of variables from $X$ of length $<\reg(\kappa)$ and $t$ is a term in $\mathcal{L}_{\mu,\kappa}$ with variables contained in $\bar x$. Let $\mathfrak{t}_1$ be the subset of $\mathfrak{t}_0$ consisting of pairs of the form $(F_{i,|\bar v_i|}(t_0,\dots, t_{|\bar v_i|-1}),\bar x)$, where $i<\omega$, and $(t_0,\bar x),\dots, (t_{|\bar v_i|-1},\bar x)\in \mathfrak{t}_0$. Let $\mathfrak{t}=\{((s_0,\bar x_0),\dots,(s_{r-1},\bar x_{r-1}))\in (\mathfrak{t}_1)^r: \bar x_0=\ldots=\bar x_{r-1}\}$. We may enumerate $\mathfrak{t}=\{\bar s_i(\bar x_i):i<\nu\}$, where for ease of notation we write $(\bar s,\bar x)$ as $\bar s(\bar x)$.

Since $\mathcal{V}$ is covered by the union of $\{\bar s_i(\bar a):\bar a\in A^{\underline{|\bar x_i|}}\}_{i<\nu}$, $\mathcal{V}=\bigcup_{i< \nu} \mathcal{V}_i$, where $\mathcal{V}_i=\{\bar s_i(\bar a):\bar a\in A^{\underline{|\bar x_i|}}\}\cap \mathcal{V}$.

By Lemma \ref{L:basic-prop-chi}(1), assumption $(1)$ and since $\kap$ is regular, there exists some $i<\nu$  with $\chi(\mathcal{G}_i)\geq \kap$, where $\mathcal{G}_i=(\mathcal{V}_i,\mathcal{E}\restriction \mathcal{V}_i\times \mathcal{V}_i)$.

Set $\bar s=\bar s_i$ and $\bar x=\bar x_i$. Assume, for simplicity, that \[\bar s(\bar x)=(F_{0,k_0}(t_{0,0}(\bar x),\dots, t_{0,k_0-1}(\bar x)),\dots, F_{r-1,k_{r-1}}(t_{{r-1},0}(\bar x),\dots, t_{{r-1},k_r-1}(\bar x))).\]
\begin{claim}
$\bar s$ defines a surjective function $A^{\underline{|\bar x|}}\to \mathcal{V}_i$.
\end{claim}
\begin{claimproof}
Since $\mathcal{G}_i$ is non-empty, there exists $\bar a\in A^{\U{|\bar x|}}$ such that $\bar s(\bar a)\in \mathcal{V}$. Let $\bar b\in A^{\U{|\bar x|}}$. By Definition \ref{D:homog-repres}(1), $F_{i,k_i}(t_{i,0}(\bar b),\dots, t_{i,k_i-1}(\bar b))\in \mathcal{U}$ for all $i<r$. Note that $|\bar x|<|A|$ by assumption (3) and so there exists a permutation $\pi$ of $A$ mapping $\bar a$ to $\bar b$, and let $\widehat \pi$ be induced automorphism of $\mathcal{M}_{\mu,\kappa}$. Thus $\widehat \pi(\bar s(\bar a))=\bar s(\bar b)$. Since $V$ is $\emptyset$-definable and $\Psi_1(\bar s(\bar a))\in V$, Definition \ref{D:homog-repres}(2) gives that $\tp^M(\Psi_1(\bar s(\bar b))=\tp^M(\Psi_1(\bar s(\bar a))$ and hence $\Psi_1(\bar s(\bar b)\in V$ as well. Consequently, $\bar s$ defines a function. Surjectivity is straightforward.
\end{claimproof}
%
Let $R=\bar{s}^{-1}(\mathcal{E}\restriction \mathcal{V}_i\times \mathcal{V}_i)$ be the edge relation $\bar s$ induces on $A^{\U{|\bar x|}}$.

By assumptions $(2,3)$, in order to apply Proposition \ref{P:finding-a-shift-graph-no-order}, we are left to verify assumption $(3)$ of Proposition \ref{P:finding-a-shift-graph-no-order}. 

Let $\bar a,\bar b,\bar c,\bar d\in A^{\underline{|\bar x|}}$ satisfying $\bar a \R \bar b$ and $f_{\bar a,\bar b}=f_{\bar c,\bar d}$. The latter condition implies that the coordinate-wise map sending $\bar a\bar b$ to $\bar c\bar d$ is well defined and injective. Since $|\bar x|< |A|$, we may find a permutation $\pi$ of $A$ which maps $\bar a\bar b$ to $\bar c\bar d$. This permutation lifts to an automorphism $\widehat \pi$ of the free algebra, with $\widehat \pi(\bar s(\bar a))= \bar s(\bar c)$ and $\widehat \pi(\bar s(\bar b))=\bar s(\bar d)$. 

 Thus $\widehat \pi(t_{i,j}(\bar a))=t_{i,j}(\bar c)$ and $\widehat \pi(t_{i,j}(\bar b))=t_{i,j}(\bar d)$, for $i<r$ and $j<k_i$. By Definition \ref{D:homog-repres}(2), 
\[\tp^M((\Phi^{-1}(t_{i,j}(\bar a)))_{i<r, j<k_i},(\Phi^{-1}(t_{i,j}(\bar b)))_{i<r, j<k_i})=\]
\[\tp^M((\Phi^{-1}(t_{i,j}(\bar c)))_{i<r, j<k_i},(\Phi^{-1}(t_{i,j}(\bar d)))_{i<r, j<k_i}),\] and consequently
\[\tp^M(\Psi(\bar s(\bar a)),\Psi(\bar s(\bar b)))=\tp^M(\Psi(\bar s(\bar c)),\Psi(\bar s(\bar d))).\]
Since $\bar a \R \bar b$, $\bar s(\bar a) \mathrel{\mathcal{E}} \bar s(\bar b)$ and so $\Psi(\bar s(\bar a))\mathrel{\E}\Psi(\bar s(\bar b))$. As this is specified by the type of the pair, \[\Psi(\bar s(\bar c))\mathrel{E}\Psi(\bar s(\bar d))\] as well. As a result, $\bar s(\bar c)\mathrel{\mathcal{E}} \bar s(\bar d)$ and $\bar c \R \bar d$.

By Proposition \ref{P:finding-a-shift-graph-no-order}, there exists $m\in \mathbb{N}$ and an injective homomorphism from $\Sh_m(\omega)$ to $A^{\U{|\bar x|}}$. By composing with $\bar s$ and $\Psi$ and applying Proposition \ref{P:homomorphism-is-enough}, there exists $n\leq m$ such that $G$ contains all finite subgraphs of $\Sh_n(\omega)$.
\end{proof}

If $T$ is a countable $\omega$-stable theory and $M\models T$, one may find an injective representation, in the sense of Remark \ref{R:original-rep}, of $M$ in $\mathcal{M}_{\aleph_0,\aleph_0}(A)$, for some set $A$, see \cite[Theorem 4.4]{919}. Similarly superstable theories may be represented in $\mathcal{M}_{2^{|T|},\aleph_0}(A)$, for some set $A$, see \cite[Theorem 2.1]{superstable}. However, we may not apply the previous proposition to these representations since they may not homogeneous. We build such homogeneous representations for stable theories in which every type is stationary.

\begin{definition}
We say that a theory $T$ is \emph{stationary} if all types (over any set) are stationary.
\end{definition}

\begin{remark}
Rothmaler studies stationarity of modules in \cite{stationary}, e.g. he gives a complete description of stationary abelian groups in \cite[Theorem 4(ii)]{stationary}.
\end{remark}

\begin{fact}\cite[Lemma 2, Theorem 1]{stationary}\label{F:stationarity}
Let $T$ be a stable theory. The following are equivalent:
\begin{enumerate}
\item $T$ is stationary;
\item for any $A$, every formula which is almost over $A$ is over $A$;
\item all $1$-types over (over any set) are stationary.
\end{enumerate}
\end{fact}

\begin{proposition}\label{P:"reduct" is stationary}
Let $T$ be a complete stationary stable theory in a language $\mathcal{L}$. For every sublanguage $\mathcal{L}_0\subseteq \mathcal{L}$ there is some $\mathcal{L}_0\subseteq \mathcal{L}^\prime\subseteq \mathcal{L}$ with $|L^\prime|=|\mathcal{L}_0|+\aleph_0$ such that $T\restriction \mathcal{L}^\prime$ is stationary.
\end{proposition}
\begin{proof}
Recall that a formula $\varphi(x,d)$ is almost over $A$ if there exists an equivalence relation with finitely many classes $E(x,x^\prime)$ over $A$ such that $\forall x\forall x^\prime(E(x,x^\prime)\to (\varphi(x,d)\leftrightarrow \varphi(x^\prime,d)))$. Equivalently, for every equivalence relation with finitely many classes $E(x,x^\prime)$ over $A$ every class of $E$ is definable over $A$.

\begin{claim}
For every $\psi(x,x^\prime,z)$ and $n<\omega$ there are finitely many formulas $\theta_i(x,z)$ ($i<k$) such that
\begin{itemize}
\item[($\dagger$)] for any $z$-tuple $c$ such that $\psi(x,x^\prime,c)$ defines an equivalence relation with $\leq n$ classes, and for any $x^\prime$-tuple $d$ 
there is some $i<k$ such that $\psi(x,d,c)$ is equivalent to $\theta_i(x,c)$.
\end{itemize}
\end{claim}
\begin{claimproof}
Note that ($\dagger$) is a first order sentence.

Suppose not and fix $\psi(x,x^\prime,z)$ and $n<\omega$. This means that for every finite collection of formulas $\theta_i(x,z)$ ($i<k$) there are some $c$ and $d$ witnessing the failure of ($\dagger$). Let $\Gamma(x^\prime,z)$ be 
\[\{\text{$\psi(-,-,z)$ defines an equivalence relation with $\leq n$ classes}\}\cup\]
\[\{\exists x \neg (\psi(x,x^\prime,z)\leftrightarrow \theta(x,z)): \text{ $\theta(x,z)$ any formula}\}.\]
By assumption, $\Gamma$ is consistent. Let $(d,c)\models \Gamma (x^\prime,z)$. Then $\psi(x,d,c)$ is almost over $c$ but not over $c$, contradiction.
\end{claimproof}

Now, let $\mathcal{L}_0\subseteq \mathcal{L}$ be a sublanguage. We construct an increasing sequence  of languages $\mathcal{L}_m$ as follows. The language $\mathcal{L}_0$ is given. Assume we have constructed $\mathcal{L}_m$. For any $\psi(x,x^\prime,z)$ in the language $\mathcal{L}_{m}$ and $n<\omega$ let $\{\theta_{\psi,n,i}(x,z)\}_{i<k_{\psi,n}}$ be a finite set of formulas satisfying $(\dagger)$ (such a set exists by the claim). Let $\mathcal{L}_{m+1}=\mathcal{L}_{m}\cup \{\text{the symbolds in the formula $\theta_{\psi,n,i}$}: \psi\in \mathcal{L}_{m},\, n<\omega, i<k_{\psi,n}\}$. Now set $\mathcal{L}^\prime=\bigcup_{m<\omega} \mathcal{L}_m$. It follows that $T\restriction \mathcal{L}^\prime$ is stationary by Fact \ref{F:stationarity}(2).
\end{proof}

We leave the proof of the following easy lemma to the reader.
\begin{lemma}\label{L:stationarity-implies-not-algebraicity}
If $T$ is a complete stationary stable theory then $\dcl(A)=\acl(A)$ for any set $A$.
\end{lemma}

\begin{remark}
If $T$ is a stable theory, then, since for any $A$ every type over $\acl^{eq}(A)$ is stationary, if $T$ is eliminates imaginaries and has no algebraicity (i.e. $\acl(A)=\dcl(A)$ for any $A$) then $T$ is stationary. However, as the theory of the infinite set shows, the other direction is not true (it does not eliminate imaginaries).
\end{remark}

We will need the following lemma, which is a consequence of \cite[Lemma III.3.10]{classification}, but for the convenience of the reader we give a direct proof. Recall the definition of $\kappa(T)$ from \cite[Definition III.3.1]{classification}. 
For any infinite indiscernible sequence $I$ and a set $A$, let $\lim(I/A)$ be the limit type of $I$ in $A$ (it is denoted by $\mathrm{Av}(I,A)$ in \cite{classification}), i.e. 
\[\lim(I/A)=\{\varphi(x,c):c\in A,\, \text{ $\varphi(a,c)$ holds for cofinitely many $a\in I$}\}.\] It is a consistent complete type over $A$ by stability.


\begin{lemma}\label{L:lambda-saturation}
Let $T$ be a stationary stable theory, $M$ a model and $\lambda>\kappa(T)$ a cardinal. If for every non-algebraic type $q\in S(C)$ with $|C|<\kappa(T)$ and $C\subseteq M$ there is a $C$-independent set of realizations of $q$ in $M$ of cardinality $\lambda$, then $M$ is $\lambda$-saturated.
\end{lemma}
\begin{proof}
Let $p\in S(A)$ be a complete type with $|A|<\lambda$. If $p$ is algebraic then it is realized, so we may assume that $p$ is non-algebraic. Let $C\subseteq A$ with $|C|<\kappa(T)$ be such that $p$ does not fork over $C$. By assumption, we may find a $C$-independent set $I$ of realizations of $p|C$ in $M$ (so indiscernible over $C$ by stationarity). By \cite[Lemma III.1.10(2)]{classification}, $\lim(I/A)=p$. By \cite[Corollary III.3.5(1)]{classification}, there is $I_0\subseteq I$ with $I\setminus I_0$ indiscernible over $A$ and $|I_0|\leq \kappa(T)+|A|<\lambda$. In particular, $|I\setminus I_0|\geq \aleph_0$ and thus for every $c\in I\setminus I_0$, $p=\tp(c/A)$.
\end{proof}

We fix the following notation for the rest of the section. Let $T$ be complete stable theory, and let  $\mathbb{U}$ be a monster model. Let $\kappa=\kappa_r(T)$, i.e. $\kappa=\kappa(T)^+$ if $\kappa(T)$ is singular or $\kappa (T)$ if not (for the sake of the following, one can also take $\kappa=|T|^+$) and let $\mu=\mu^{<\kappa}$ be a cardinal, with $\mu>\kappa$, such that $T$ is $\mu$-stable, e.g. if $\mu\geq 2^{|T|}$ (see \cite[Lemma III.3.6]{classification}), and thus there exists a saturated model of cardinality $\mu$ \cite[Theorem III.3.12]{classification}. Fix some partition $\mu=\cupdot_{i<\kappa} U_i$ to sets each of cardinality $\mu$.
From now on we also assume that $T$ is stationary.

\begin{definition}\label{D:ob}
Let $I$ be any set. We define $OB(I)$ to be the collection of triples
\[\textbf{a}:=\left( i_\Ba,\{U_j^\Ba\}_{j<i_\Ba}, B^\Ba\right)\] satisfying:
\begin{enumerate}
\item $i_\Ba\leq \kappa$;
\item $U_j^\Ba\subseteq U_j$ for all $j<i_\Ba$, and we set $U_{<j}^\Ba:=\bigcup_{k<j}U^\Ba_k$;
\item $B^\Ba=\langle b^\Ba_{\alpha,\eta}: \alpha\in U_{j}^\Ba, \eta\in I^{\U{j}},\, j<i_\Ba\rangle$ are such that:
\begin{enumerate}
\item $B^\Ba$ is with no repetitions;
\item $B^\Ba_j:=\{ b^\Ba_{\alpha,\eta}: \alpha\in U_j^\Ba,\eta\in I^{\U j}\}$ is independent over $B_{<j}^\Ba=\bigcup_{k<j}B_k^\Ba$, for every $j<i_\Ba$.

\end{enumerate}

For ease, we denote for $j<i_\Ba$, $W^\Ba_j:=\{(\alpha,\eta):\alpha\in U_j^\Ba, \eta\in I^{\U{j}}\}$, and likewise $W^\Ba_{<j}$.

\end{enumerate}

As usual, when $\Ba$ is clear from the context we omit it.
\end{definition}

Note that any permutation $\pi$ of the set $I$ induces a permutation $\widehat \pi$ of $I^{\U \gamma}$, for any $\gamma$.

\begin{definition}
Let $I$ be a set and $OB(I)$ as above.
\begin{enumerate}
\item  We say that $\Ba\in OB(I)$ is \emph{homogeneous} if for any permutation $\pi$ of $I$,
the set of pairs \[\pi[\Ba]:=\{(b^\Ba_{\alpha,\eta},b^\Ba_{\alpha,\widehat \pi(\eta)}):(\alpha,\eta)\in W_{<i_\Ba}^\Ba\}\] is an elementary embedding.

\item We say that $\Ba\in OB(I)$ is \emph{full} if for every $j<i_{\Ba}$ and $\eta\in I^{\U j}$,  a non-algebraic type $p$ over $B_{<j}^\Ba$ which does not fork over \[\{b^\Ba_{\alpha,\nu}\in B^\Ba_{<j}: \Rg(\nu)\subseteq \Rg(\eta)\}\] is realized by $b^\Ba_{\alpha,\eta}$ for some $\alpha\in U_j^\Ba$.
\end{enumerate}
\end{definition}

\begin{lemma}\label{L:use-fullness}
Let $I$ be any set and $\Ba\in OB(I)$ with $i_{\Ba}=\kappa$. For every $C\subseteq B^{\Ba}_{<\kappa}$ with $|C|<\kappa(T)\leq \kappa$ there exist some $j<\kappa$ and $\eta\in I^{\U{j}}$ satisfying
\[C\subseteq \{b_{\alpha,\nu}\in B^\Ba_{<j}:\Rg(\nu)\subseteq \Rg(\eta)\}.\]
\end{lemma}
\begin{proof}
Since $\kappa$ is regular, there exist $\tilde{j}<\kappa$ with $C\subseteq B^\Ba_{<\tilde{j}}$. Let $J=\bigcup_{b_{\alpha,\nu}\in C}\Rg(\nu)$.
Let $j=\max\{\tilde{j},|J|\}<\kappa$ and $\eta\in I^{\U j}$ with $\Rg(\eta)=J$.
\end{proof}

\begin{proposition}\label{P:object is saturated}
Let $I$ be any set with $|I|\geq \mu$. If $\Ba\in OB(I)$ is full and $i_\Ba=\kappa$ then $M:=\mathbb{U}\restriction \dcl(B_{<\kappa}^{\Ba})$ is a saturated elementary substructure of $\mathbb{U}$ of cardinality $|I|^{<\kappa}$.
\end{proposition}
\begin{proof}
To show that it is an elementary substructure we use Tarski-Vaught. Let $\varphi(x,b)$ be a consistent formula with $b\in \dcl(B_{<\kappa}^{\Ba})$. There is no harm in assuming that $b\in B_{<\kappa}^{\Ba}$. If $\varphi(x,b)$ is algebraic then by Lemma \ref{L:stationarity-implies-not-algebraicity}, any realization is already in $\dcl (B_{<\kappa}^{\Bb})$. Otherwise, let $p$ be any non-algebraic complete type over $b$ containing $\varphi(x,b)$. Let $j<\kappa$ and $\eta\in I^{\U j}$ be given by Lemma \ref{L:use-fullness} for $C=\{b\}$. By stationarity, there is a unique non forking extension $\mathfrak{p}$ of $p$ to $B^\Ba_{<j}$, which is necessarily non-algebraic as well. By fullness, we may realize $\mathfrak{p}$ by some element $b_{\alpha,\eta}$ for $\alpha\in U_j^{\Ba}$. In particular $b_{\alpha,\eta}\in B_{<\kappa}^\Ba$ realizes $\varphi(x,b)$.

To show $|I|^{<\kappa}$-saturation we apply Lemma \ref{L:lambda-saturation} (recall $\mu>\kappa$). 
Let $q\in S(C)$ be a non-algebraic type with $C\subseteq \dcl(B^{\Ba}_{<\kappa})$ and $|C|<\kappa(T)$. There is no harm to take $C\subseteq B^{\Ba}_{<\kappa}$. Let $j<\kappa$ and $\eta\in I^{\U j}$ be as supplied by Lemma \ref{L:use-fullness} with respect to $C$.  Let 
\[\Delta:=\{\nu\in \bigcup_{j\leq k<\kappa} I^{\U k}: \Rg(\eta)\subseteq \Rg(\nu)\}.\]
By assumption of fullness, for every $\nu\in \Delta$ there is some $\alpha_\nu\in U^{\Ba}_{lg(\nu)}$, such that $b_{\alpha_\nu,\nu}\models q|B_{<lg(\nu)}^{\Ba}$, and so satisfies $q$ as well. The set of realizations of $q$, $\{ b_{a_\nu,\nu}:\nu \in \Delta\}$, is independent over $C$ by the definition of $OB(I)$. Indeed, by  Definition \ref{D:ob}(3b) we show by induction on $j\leq k$ that $\{b_{\alpha_\nu,\nu}:\nu\in \Delta,\,  lg(\nu)<k\}$ is independent over $C$. 

We note that $|\Delta|=|I|^{<\kappa}$ and thus by Lemma \ref{L:lambda-saturation}, $\dcl(B_{<\kappa}^\Ba)$ is $|I|^{<\kappa}$-saturated. Furthermore, as $|I|\geq \mu$, it follows that $|I|^{<\kappa}=|B^{\Ba}_{<\kappa}|=|\dcl(B^{\Ba}_{<\kappa})|$ and so $\dcl(B^{\Ba}_{<\kappa})$ is saturated.
\end{proof}

We define the following partial order on the elements of $OB(I)$.
\begin{definition}
Let $I$ be a set. For $\Ba,\Bb\in OB(I)$ we say that $\Ba\leq \Bb$ if 
\begin{enumerate}
\item $i_\Ba\leq i_\Bb$;
\item for $j<i_\Ba$ we have $U_j^\Ba=U_j^\Bb$ and $b_{\alpha,\eta}^\Ba=b_{\alpha,\eta}^\Bb$, for $(\alpha,\eta)\in W_j^\Ba$.
\end{enumerate}
\end{definition}

\begin{proposition}\label{P:existence of object}
Let $I$ be a set with $|I|\geq\kappa$. Then there exists a full homogeneous $\Ba\in OB(I)$ with $i_\Ba=\kappa$.
\end{proposition}
\begin{proof}
We choose full homogeneous $\Ba_j\in OB(I)$ by induction on $j\leq \kappa$ such that $\Ba_j\in OB(I)$ with $i_{\Ba_j}=j$ and such that $k_1\leq k_2<j$ implies $\Ba_{k_1}\leq \Ba_{k_2}$.

For $j=0$ choose $\Ba_0=(0,\emptyset,\emptyset)$ and note that the conditions hold trivially.

Let $j\leq \kappa$ be a limit ordinal, set $i_{\Ba_{j}}=j$, $U_k^{\Ba_j}:=U_k^{\Ba_{k+1}}$ for $k<j$ and $B^{\Ba_j}:=\bigcup_{k<j}B^{\Ba_k}$, $\Ba_j$ has the desired properties.

Let $j<\kappa$ and assume that $\Ba:=\Ba_j$ is full and homogeneous with $i_{\Ba_j}=j$. We will construct a full homogeneous $\Bb\in OB(I)$ with $i_\Bb=j+1$ such that $\Ba\leq \Bb$. Then we can set $\Ba_{j+1}:= \Bb$.

Let \[\mathcal{P}:=\{(p,C,\eta): p\in S(C) \text{ non-algebraic}, \eta\in I^{\U j},\]\[ C\subseteq \{b_{\alpha,\nu}: b_{\alpha,\nu}\in B_{<j}^{\Ba}, \Rg(\nu)\subseteq \Rg(\eta)\} \text{ of cardinality $< \kappa$}\}.\] 
We define the following equivalence relation $\mathcal{P}$:
\[(p_1,C_1,\eta_1)\E(p_2,C_2,\eta_2)\]
if and only if for some permutation $\pi$ of $I$ 
\[\pi\langle \Ba\rangle(p_1,C_1,\eta_1)=(p_2,C_2,\eta_2),\]
which means $\widehat \pi(\eta_1)=\eta_2$, $\pi[\Ba]$ (as an elementary mapping) maps $C_1$ onto $C_2$ and $p_1$ onto $p_2$. Since $|I|\geq \kappa>j$, for every $\eta_1,\eta_2\in I^{\U j}$ there is a permutation $\pi$ of $I$ mapping $\eta_1$ to $\eta_2$. As $\pi\langle \Ba\rangle$ is a permutation of $\mathcal{P}$, $E$ has at most $\mu$ equivalence classes (because $\mu^{<\kappa}=\mu$).

Since $|U_j|=\mu$ we may find some $U_j^\prime\subseteq U_j$ that enumerates the different classes, i.e let
$\langle X_\alpha:\alpha\in U_j^\prime\subseteq U_j\rangle$ list  $\mathcal{P}/\E$. Since $\pi\langle \Ba\rangle$ is a permutation of $\mathcal{P}$, $\pi\langle \Ba\rangle\restriction X_\alpha$, for $\alpha\in U_j^\prime$, is also a permutation.

For any $\eta\in I^{\U j}$ and $\alpha\in U_j^\prime$ there exists a unique $(p,C)$ such that $(p,C,\eta)\in X_\alpha$. To show uniqueness, note that if $\widehat \pi$ fixes $\eta$ then $\pi[\Ba]$ fixes $C$. For existence, let $\nu\in I^{\U j}$, $p$ and $C$ be such that $(p,C,\nu)\in X_\alpha$. Since $|I|>j$ we may find a permutation $\pi$ of $I$ mapping $\nu$ to $\eta$. Now note that  $(\pi[\Ba](p), \pi[\Ba](C), \widehat \pi (\nu))\in X_\alpha$. 

For any $\eta\in I^{\U j}$ and $\alpha\in U_j^\prime$ we name the unique pair by $(p_{\alpha,\eta},C_{\alpha,\eta})$. By fixing some well order on $\{(\alpha,\eta): \alpha\in U_j^\prime, \eta\in I^{\U j}\}$, we may inductively find $\langle b_{\alpha,\eta}:\alpha\in U_j^\prime, \eta\in I^{\U j}\rangle$ such that $b_{\alpha,\eta}$ realizes the unique non-forking extension of $p_{\alpha,\eta}$ to $B^{\Ba}_{<j}\cup \{b_{\alpha^\prime,\eta^\prime}: (\alpha^\prime,\eta^\prime)<(\alpha,\eta)\}$. In particular $\tp(b_{\alpha,\eta}/B^{\Ba}_{<j})$ does not fork over $C_{\alpha,\eta}$ and $\langle b_{\alpha,\eta}:\alpha\in U_j^\prime, \eta\in I^{\U j}\rangle$ is independent over $B^{\Ba}_{<j}$. 

We may now define $\Bb\in OB(I)$ by $i_\Bb=j+1$, $\Ba\leq \Bb$, $U^\Bb_j=U^\prime_j$ and $b_{\alpha,\eta}^\Bb=b_{\alpha,\eta}$ for $\alpha\in U_j^\prime, \eta\in I^{\U j}$. Note that since the types $p_{\alpha,\eta}$ are not algebraic, it follows that $B^{\Bb}$ is without repetitions. It remains to check that $\Bb$ is full and homogeneous.

$\Bb$ is full: by induction hypothesis, we consider $\eta\in I^{\U j}$ and a non-algebraic type $p$ over $B_{<j}^\Bb$ which does not fork over
\[\{b_{\alpha,\nu}\in B^{\Bb}_{<j}: \Rg(\nu)\subseteq \Rg(\eta)\}.\]
Let $C$ be a subset with $|C|<\kappa$ such that $p$ does not fork over $C$. Consider the triple $(p|C,C,\eta)$ and let $\alpha\in U_j^\Bb$ be such that $(p|C,C,\eta)\in X_\alpha$. By uniquness, $(p|C,C)=(p_{\alpha,\eta},C_{\alpha,\eta})$ (in the above notation), and thus $b_{\alpha,\eta}^\Bb$ satisfies the unique non-forking extension of $p|C$ to $B_{<j}^\Bb$ which is equal to $p$.

$\Bb$ is homogeneous: let $\pi$ be a permutation of $I$. 
To show that $\pi[\Bb]$ is elementary, we show by induction on $k\geq 0$ that for $(\alpha_1,\eta_1),\dots, (\alpha_k,\eta_k)\in W_{j}^{\Bb}$
\[b_{\alpha_1,\eta_1}\dots b_{\alpha_k,\eta_k}B^{\Bb}_{<j}\equiv b_{\alpha_1,\widehat\pi(\eta_1)}\dots b_{\alpha_k,\widehat\pi(\eta_k)}\pi[\Bb](B^{\Bb}_{<j}).\]
For $k=0$, it follows since $\pi[\Ba]$ is elementary and since $B^{\Bb}_{<j}=B^{\Ba}_{<j}$.

For $k=1$: since $\pi[\Bb]$ is an elementary map mapping $p_{\alpha,\eta}$ onto $p_{\alpha,\widehat \pi(\eta)}$, by stationarity it also maps $p_{\alpha,\eta}|B^{\Bb}_{<j}$ onto $p_{\alpha,\widehat \pi(\eta)}|B^{\Bb}_{<j}$ and so $b_{\alpha,\eta}B^{\Bb}_{<j}\equiv b_{\alpha,\widehat \pi(\eta)}\pi[\Bb](B^{\Bb}_{<j})$. 

The induction step: let $(\alpha_1,\eta_1),\dots, (\alpha_k,\eta_k)\in W_{j}^{\Bb}$ with $k\geq 2$. By the induction hypothesis \[b_{\alpha_1,\eta_1}\dots b_{\alpha_{k-1},\eta_{k-1}}B^{\Bb}_{<j}\equiv b_{\alpha_1,\widehat\pi(\eta_1)}\dots b_{\alpha_{k-1},\widehat\pi(\eta_{k-1})}\pi[\Bb](B^{\Bb}_{<j})\] and  
$b_{\alpha_k,\eta_k}B^{\Bb}_{<j}\equiv b_{\alpha_k,\widehat \pi (\eta_k)}\pi[\Bb](B^{\Bb}_{<j}).$

Moreover, since the elements are independent, \[b_{\alpha_k,\eta_k}\forkindep[B^{\Bb}_{<j}] b_{\alpha_1,\eta_1}\dots b_{\alpha_{k-1},\eta_{k-2}} \text{ and }\]
\[b_{\alpha_k,\widehat \pi(\eta_k)}\forkindep[{\pi[\Bb](B^{\Bb}_{<j})}] b_{\alpha_1,\widehat\pi(\eta_1)}\dots b_{\alpha_{k-1},\widehat\pi(\eta_{k-2})}\] and consequently by stationarity (see also Lemma \ref{L:stationarity} below)
\[b_{\alpha_1,\eta_1}\dots b_{\alpha_k,\eta_k}B^{\Bb}_{<j}\equiv b_{\alpha_1,\widehat\pi(\eta_1)}\dots b_{\alpha_k,\widehat\pi(\eta_k)}\pi[\Bb](B^{\Bb}_{<j}).\]
As required.
\end{proof}

\begin{theorem}\label{T:existence-of skel}
Let $T$ be a complete stationary stable theory. Let $\kappa=\kappa_r(T)$ and $M\models T$ a saturated model of cardinality $\geq \mu=\mu^{<\kappa}$ such that $\mu\geq 2^{|T|}$ (so $T$ is $\mu$-stable) and $\kappa<\mu$. Let $I$ be any set such that $|I|^{<\kappa}=|M|$.

Then there exists a skeletal homogeneous representation of $M$ in $\mathcal{M}_{\mu,\kappa}(I)$. In fact, the representation will be in $\mathcal{M}_{\mu,\kappa,1}(I)$.
\end{theorem}
\begin{proof}
Let $I$ be any set such that $|I|^{<\kappa}=|M|\geq \mu$ and let $\Ba\in OB(I)$  be a full homogeneous object with $i_\Ba =\kappa$ as supplied by Proposition \ref{P:existence of object}. By Proposition \ref{P:object is saturated}, $\dcl(B_{<\kappa}^\Ba)\models T$ is saturated of cardinality $|I|^{<\kappa}=|M|$. In particular $M$ is isomorphic to $\dcl(B_{<\kappa}^\Ba)$. Without loss of generality we assume $M=\dcl(B_{<\kappa}^\Ba)$. 
We define a function
\[\Phi:B_{<\kappa}^\Ba\to \mathcal{M}_{\mu,\kappa,1}(I)\]
by $\Phi(b_{\alpha,\eta})=F_{\alpha,j}(\eta)$ for the unique $j<\kappa$ such that $\eta\in I^{\U j}$ and $\alpha\in U_j^{\Ba}$.

The map $\Phi$ is injective. If $F_{\alpha,j}(\eta)=F_{\beta,k}(\nu)$ then, since it is a free algebra, $\alpha=\beta, j=k$ and $\eta=\nu$ so $b_{\alpha,\eta}=b_{\beta,\nu}$.

The map $\Phi$ is a homogeneous representation. Indeed, to show condition (1), let $t(\bar x)$ be any term with $|\bar x|=\beta<\kappa$ and let $\bar a\in I^{\U \beta}$ with $t(\bar a)\in \Img (\Phi)$. Thus there exist $j<\kappa$ and $\alpha\in U_j^{\Ba}$ such that $t(\bar a)=F_{\alpha,j}(\eta_{\bar a})$ for some $\eta_{\bar a}\in I^{\U j}$. Thus for any $\bar b\in I^{\U \beta}$ there is some $\eta_{\bar b}\in I^{\U j}$ with $t(\bar b)=F_{\alpha,j}(\eta_{\bar b})$. In particular, $t(\bar b)\in \Img(\Phi)$, as needed.

For condition (2), let $b_{\alpha_1,\eta_1}\dots b_{\alpha_k,\eta_k}\in (B_{<\kappa}^{\Ba})^{\U k}$ and let $\pi$ be a permutation of $I$. Since $\Ba$ is homogeneous $\tp(b_{\alpha_1,\eta_1}\dots b_{\alpha_k,\eta_k})=\tp(b_{\alpha_1,\widehat\pi(\eta_1)}\dots b_{\alpha_k,\widehat\pi(\eta_k)})$, as needed. 
\end{proof}

\begin{corollary}\label{C:stationarystable}
Let $G=(V,E)$ be a graph that is interpretable (possibly with paramters) in a stationary stable structure. If $\chi(G)> \beth_2(\aleph_0) $ then there exists an $n\in\mathbb{N}$ such that $G$ contains all finite subgraphs of $\Sh_n(\omega)$.
\end{corollary}
\begin{proof}
Assume $G$ is interpretable in a stationary stable structure $N$ over some finite set of a parameters $A\subseteq N$ and let $T=Th(N)$. Since adding constants to the language preserves stationarity, we may assume that $G$ is interpretable in $N$ over $\emptyset$. Since the interpretation only uses a finite fragment of the language, by applying Proposition \ref{P:"reduct" is stationary} we may assume that $|T|=\aleph_0$. 

Let $\mu=2^{\aleph_0}$ and $\kappa=\kappa(T)$. Note that $\kappa(T)\leq \aleph_1$ (\cite[Corollary III.3.3]{classification}) which implies $\kappa(T)=\kappa_r(T)$ and $\mu^{<\kappa}=\mu$. Let $I$ be any set satisfying $|I|\geq \max\{\mu,|N|\}$ (which implies $|I|^{<\kappa}\geq \max\{\mu,|N|\}$) and let $M\models T$ be a saturated elementary extension of $N$ of cardinality $|I|^{<\kappa}$ (exists by \cite[Lemma III.3.6 and Theorem III.3.12]{classification}). 

By Theorem \ref{T:existence-of skel} and Proposition \ref{P:embed-shift-represen}, there exists $n\in\mathbb{N}$ such that $(G(M),E(M))$, the realizations in $M$ of the interpretation of $G$, contains all finite subgraphs of $\Sh_n(\omega)$ and since $N\prec M$ the result follows. 
\end{proof}

A natural question is whether every stable structure is interpretable in a stationary stable structure. We thank Hrushovski for the following argument.

\begin{proposition}\label{P:stationary-Udi}
Let $T$ be a stationary stable theory. Then $T$ does not interpret an infinite $p$-root closed field, where $p$ is a prime number different from $\mathrm{char}(F)$. In particular, ACF is not interpretable in any stationary stable theory.
\end{proposition}
\begin{proof}
Let $M\models T$ and assume towards a contradiction that it interprets an infinite field $F$, that is $p$-root closed for $p\neq\mathrm{char}(F)$. 

For ease of writing, we assume that $F$ is a definable field in $M^{eq}$. So there exists a definable set $D\subseteq M^n$, for some $n<\omega$, and a definable (in $M^{eq}$) surjective map $\pi:D\to F$. 

Let $1\neq \zeta_p\in F$ be a $p$-th root of unity (such exists since $p\neq \mathrm{char}(F)$) and for every $a_0,\dots,a_{p-1}\in D$, let $\sigma_{a_0,\dots,a_{p-1}}=\Sigma_{i=0}^{p-1}\pi(a_i)\zeta_p^i$. Note that $\zeta_p\sigma_{a_0,\dots,a_{p-1}}=\sigma_{a_{p-1},a_0,\dots,a_{p-2}}$.

For any $a_0,\dots a_{p-1}\in D$, $(\pi(x)=\pi(y)\wedge \pi(x)^p=\sigma_{a_0,\dots,a_{p-1}})\vee (\pi(x)^p\neq \sigma_{a_0,\dots,a_{p-1}}\wedge \pi(y)^p\neq\sigma_{a_0,\dots,a_{p-1}} )$ is a definable finite equivalence relation on $D$ over $a_0,\dots,a_{p-1}$ (definable in $M$). By Fact \ref{F:stationarity} each of the equivalence classes are definable over $a_0,\dots,a_{p-1}$ (see also the first paragraph of the proof Proposition \ref{P:"reduct" is stationary}).

By compactness, there is a definable function $f:D^p\to F$ satisfying \[f(a_0,\dots,a_{p-1})^p=\sigma_{a_0,\dots,a_{p-1}}.\]

Let $\mathfrak{p}$ be a non-algebraic global type on $F$ (exists since $F$ is infinite), and since $\pi$ is surjective, we may find a global type $\mathfrak{q}$ on $D$ with $\pi_*\mathfrak{q}=\mathfrak{p}$. After naming parameters, we may assume that both $\mathfrak{p}$ and $\mathfrak{q}$ are $\emptyset$-definable. 

Let $(a_0,\dots,a_{p-1})\models \mathfrak{q}^{(p)}|\zeta_p$ (where $\mathfrak{q}^{(p)}=\mathfrak{q}\otimes \mathfrak{q}^{(p-1)}$). Since $\zeta_pf(a_0,\dots,a_{p-1})^p=f(a_{p-1},a_0,\dots,a_{p-2})^p$, then letting $\omega=\frac{f(a_{p-1},a_0,\dots,a_{p-2})}{f(a_0,\dots,a_{p-1})}\in F$ we have that $\omega^p=\zeta_p$. Since $(a_{\tau(0)},\dots,a_{\tau(p-1)})\models \mathfrak{q}^{(p)}|\zeta_p$, for any permutation $\tau$ on $\{0,\dots,p-1\}$, 
\[\omega f(a_0,\dots,a_{p-1})=f(a_{p-1},a_0,\dots,a_{p-2})\]
\[\omega f(a_{p-1},a_0,\dots,a_{p-2})=f(a_{p-2},a_{p-1},a_0,\dots,a_{p-3})\]
\[\vdots\]
\[\omega f(a_1,\dots,a_{p-1},a_0)=f(a_0,\dots, a_{p-1}).\]
Thus $f(a_0,\dots,a_{p-1})=\omega^pf(a_0,\dots,a_{p-1})=\zeta_pf(a_0,\dots,a_{p-1})$. This implies that $\sigma_{a_0,\dots,a_{p-1}}=0$, contradicting the non-algebraicity of $\mathfrak{p}$.
%
%
%
\end{proof}

\begin{remark}
Since every first order infinite structure (in a finite language) is bi-interpretable with a graph \cite[Theorem 5.5.1]{hodges}, it follows that there is a stable graph that is not interpretable in any stationary stable structure.

\end{remark}

\section{Quantitative Bounds}
The following section is joint work with Elad Levi.

The aim of this section is to prove that if $G=(V,E)$ is an $\omega$-stable graph with uncountable chromatic number and the $\Ur$-rank of $G$, $\Ur(G),$ is at most $2$ then it contains all finite subgraphs of $\Sh_n(\omega)$ for some $n\leq 2$. For the definition of $\Ur$-rank see \cite[Definition 8.6.1]{TZ}. Throughout, we will use Lascar's equality when the $\Ur$-rank is finite, see \cite[Exercise 8.6.5]{TZ}.

For certain parts of the argument we will need the following assumption.
\begin{context*}\label{H:hyp}
$G$ is a saturated $\omega$-stable structure, with home sort $V$, that eliminates imaginaries in a countable language with $\acl(\emptyset)=\dcl(\emptyset)$. Let $E\subseteq V^2$ be an $\emptyset$-type-definable set. Let $p\in S_1(\emptyset)$ be a non-algebraic and of finite $\Ur$-rank such that $G_p=(p(G),E\restriction p(G))$ is a graph with $\chi(G_p)\geq \aleph_1$.
\end{context*}
%
Note that assumption $\diamondsuit$ implies that every type over $\emptyset$ is stationary.

Assume $\diamondsuit$ and let $E_{alg}=\{(a,b)\in E:a\in \acl(b)\wedge b\in\acl(a)\}$ be the set of interalgebraic pairs belonging to $E$. Note that if $\Ur (a)=\Ur(b)$ then $a\in\acl(b)$ if and only if $b\in \acl(a)$. Indeed, by Lascar's equality \[\Ur(a/b)+\Ur(b)=\Ur(ab)=\Ur(a)+\Ur(b/a)\]
and for any type $q$, $\Ur(q)=0$ if and only if it is algebraic, see \cite[Exercise 8.6.1]{TZ}. 
Let $E_{nalg}=E\setminus E_{alg}$, it is definable by a countable type. 
%

\begin{lemma}\label{L:must-be-U2}
Assume $\diamondsuit$.
\begin{enumerate}
\item $\chi(p(G),E_{nalg})\geq \aleph_1$;
\item If there exist $a,b\in G_p$ with $a\E b$ and $a\forkindep b$ then any Morley sequence based on $p$ forms an infinite complete graph;
\item If $\Ur(p)=1$ then there exists a Morley sequence based on $p$ which forms an infinite complete graph
\end{enumerate}
\end{lemma}
\begin{proof}
$(1)$ By interalgebraicity, every connected component of $(p(G),E_{alg})$ is countable and consequently $\chi(p(G),E_{alg})\leq \aleph_0$. By Lemma \ref{L:basic-prop-chi}(2), $\chi(p(G),E_{nalg})\geq \aleph_1$.

$(2)$ Assume there exist $a,b\in G_p$ with $a\E b$ and $a\forkindep b$. Since every type over $\emptyset$ is stationary it follows that every Morley sequence based on $p$ forms an infinite complete graph.

$(3)$ Assume $\Ur(p)=1$. Since $\chi(p(G),E_{nalg})\geq \aleph_1$, there must exist some $a,b\in p(G)$ with $a\E_{nalg}b$. If $a\not\forkindep b$ then $\Ur(a/b)<\Ur(a)=1$, which implies that $a\in\acl(b)$ (so $b\in\acl(a))$), contradiction. Thus $a\forkindep b$ and we may use $(1)$.
\end{proof}
%
%

\begin{definition}
We say that a stationary type $\tp(a/A)$ is \emph{pseudo-one-based} if $\mathrm{Cb}(a/A)\subseteq \acl^{eq}(a)$.
\end{definition}
\begin{remark}
Compare with the last paragraph of page 105 in \cite{pillay}.
\end{remark}
We give some examples of pseudo-one-based types. 
%
\begin{lemma}\label{L:p-o-b}
\begin{enumerate}
\item In a one-based theory every stationary type (over any base) is pseudo-one-based.
\item Let $M$ be a stable structure. If $\Ur(a)=\Ur(b)=1$ and $a\not\forkindep b$ then $\tp(a/\acl^{eq}(b))$ is pseudo-one-based.
\item Let $M$ be a stable structure and $a,b\in M$ non interalgebraic, with $\Ur(a)=\Ur(b)=2$. Let $X$ and $Y$ be infinite mutually indiscernible sets with $a\in X$ and $b\in Y$. If $a\not\forkindep b$ then $\tp(a/\acl^{eq}(b))$ is pseudo-one-based.
\end{enumerate}
\end{lemma}
\begin{proof}
$(1)$ A stable theory is one-based if for all $a,B$, $\mathrm{Cb}(a/\acl^{eq}(B))\subseteq \acl^{eq}(a)$. Note that since $\tp(a/A)$ is stationary, $\mathrm{Cb}(a/\acl^{eq}(A))=\mathrm{Cb}(a/A)$. The result follows.

$(2)$ Since $a\not\forkindep b$, by $\Ur$-rank considerations as before, $a$ and $b$ are interalgebraic. So $\mathrm{Cb}(a/\acl^{eq}(b))\subseteq \acl^{eq}(b)\subseteq \acl^{eq}(a)$.

$(3)$ We note that for any $b \neq b^\prime\in Y$, $a\forkindep[b^\prime] b$. Indeed, otherwise $\Ur(a/bb^\prime)<\Ur(a/b^\prime)<\Ur(a)$ and since $\Ur(a)=2$, $a\in \acl(bb^\prime)$, contradicting the mutual indiscernibility of $X,Y$. Similarly, $a\forkindep[b]b^\prime$. Consequently, setting $e:=\mathrm{Cb}(a/\acl^{eq}(bb^\prime))$, $e\subseteq \acl^{eq}(b)\cap \acl^{eq}(b^\prime)$ and $a\forkindep[e]bb^\prime$. Similarly we get that $b\forkindep[a] b^\prime$ and by the properties of forking $e\forkindep[a] e$ and hence $e\in\acl^{eq}(a)$. Finally, since $a\forkindep[e] b$ (and $e\in \acl^{eq}(b)$), $\mathrm{Cb}(a/\acl^{eq}(b))\subseteq \acl^{eq}(e)\subseteq \acl^{eq}(a)$, as needed. 
\end{proof}

Abundance of pseudo-one-based types will be a key tool in our proofs. The above shows that this can be achieved in one-based theories and $\Ur$-rank $1$ types. For $\Ur$-rank $2$ we observe the following:

\begin{lemma}\label{L:E_0-for-U-2}
Assume $\diamondsuit$ and that $\Ur(p)=2$. Then either we can embed an infinite complete graph into $G_p$ or there exists a type-definable symmetric irreflexive relation $E_0\subseteq E_{nalg}$ such that
\begin{itemize}
\item[($\dagger$)] for every $(a,b)\in E_0$, $\tp(a/\acl(b))$ is pseudo-one-based. Moreover, if $F\subseteq E$ is a symmetric irreflexive type-definable relation with $\chi(G_p, F)\geq \aleph_1$ then $F\cap E_0\neq \emptyset$.
\end{itemize}

\end{lemma}
\begin{proof}
Assume that we cannot embed an infinite complete graph into $G_p$, in particular by Lemma \ref{L:must-be-U2}, for every $a,b\models p$ with $a \E b$, $a\not\forkindep b$.

Let $E_0$ be the set of pairs $(a,b)\in E\restriction p(G)$ such that there exists a complete bipartite subgraph $K_{X,Y}$ of $G_p$ such that $X$ and $Y$ are infinite mutually indiscernible sets with $a\in X$ and $b\in Y$. Easily, $E_0$ is type-definable by a countable type. Since, by \cite[Corollary 5.6]{EH}, $G_p$ contains $K_{n,n}$ (the complete bipartite graph on $n$ vertices) for every $n<\omega$, $E_0$ non-empty. Clearly, $E_0\subseteq E_{nalg}$. By Lemma \ref{L:p-o-b}(3), for every $(a,b)\in E_0$, $\tp(a/\acl(b))$ is pseudo-one-based.
%

For the moreover part, if $\chi(G_p,F)\geq\aleph_1$ then by by \cite[Corollary 5.6]{EH} we may embed $K_{n,n}$ into it for any $n<\omega$. Thus by saturation necessarily $F\cap E_0\neq \emptyset$.
\end{proof}

The following is the key proposition of the proof and where pseudo-one-based types show their usefulness.

\begin{proposition}\label{P:main-prop-finite-u-rank}
Assume $\diamondsuit$ and that $E_0\subseteq E_{nalg}$ is a type-definable symmetric irreflexive relation satisfying ($\dagger$) from Lemma \ref{L:E_0-for-U-2}.

\begin{enumerate}
\item For any $(a,b)\in E_{0}$ there is a finite tuple $e$ such that $a\forkindep[e] b$ and $\tp(a/e)$ is stationary.
\item Let $\Psi$ be the collection of all pairs of formulas $(\varphi(u,x),\psi(u,x))$ satisfying
\begin{enumerate}
\item $\varphi(u,a)$ and $\psi(u,a)$ are algebraic formulas each isolating a complete type over some (any) $a\models p$;
\item there exist $a,b\models p$ and $e$ such that $(a,b)\in E_0$, $a\forkindep[e] b$, $\varphi(e,a)$, $\psi(e,b)$ and $\tp(a/e)$ is stationary. 
\end{enumerate}
For any $(\varphi,\psi)\in\Psi$ let \[E_{\varphi,\psi}=\{(a,b)\in E_{nalg}:a,b\in p(G),\, \exists e\left(\varphi(e,a)\wedge \psi(e,b)\right)\}.\]

Then either $(*)$ we can embed an infinite complete graph into $(p(G),E_{nalg})$ or there exists $(\varphi(u,x),\psi(u,x))\in \Psi$ such that $(**)_{\varphi,\psi}$: $p\vdash \forall u(\varphi(u,x)\rightarrow \neg \psi(u,x))$ and $\chi(G_{p,\varphi,\psi})\geq\aleph_1$, where $G_{p,\varphi,\psi}:=(p(G),E_{\{\varphi,\psi\}})$ and $E_{\{\varphi,\psi\}}=E_{\varphi,\psi}\vee E_{\psi,\varphi}$.

\item Assume $(**)_{\varphi,\psi}$. There exists $q_\varphi\in S(\emptyset)$ such that for any $a \models p$ and $e\models \varphi(u,a)$, $e\models q_\varphi$. Similarly, there exists $q_\psi\in S(\emptyset)$ such that for any $a \models p$ and $e\models \psi(u,a)$, $e\models q_\psi$. Furthermore, $q:=q_\varphi=q_\psi$ and $0<\Ur(q)<\Ur(p)$.

\item Assume $(**)_{\varphi,\psi}$ and let $q=q_\varphi=q_\psi$. The type-definable relation $e_1 \R e_2$ given by \[(\exists a\models p) \left( (\varphi(e_1,a)\wedge \psi(e_2,a)\vee (\varphi(e_2,a)\wedge \psi(e_1,a))\right),\]
defines a graph $H_q$ on realizations of $q$ with $\chi(H_q)\geq \aleph_1$.

\end{enumerate}
\end{proposition}
\begin{proof}
We start by showing that for $(a,b)\in E$,
\[E_0\subseteq \bigcup_{(\varphi,\psi)\in \Psi} E_{\{\varphi,\psi\}}.\]
Let $(a,b)\in E_0$. Since $\tp(a/\acl(b))$ is pseudo-one-based, $\mathrm{Cb}(a/\acl(b))\subseteq \acl(a)\cap\acl(b)$. By \cite[Exercise 8.4.7]{TZ}, there is a finite tuple $e$ such that $\dcl(e)=\mathrm{Cb}(a/\acl(b))$ (this proves $(1)$).  We choose $\varphi(u,a)$ to be a formula isolating $\tp(e/a)$ and $\psi(u,b)$ to be a formula isolating $\tp(e/b)$. Hence $(a,b)\in E_{\{\varphi,\psi\}}$ and $(\varphi,\psi)\in \Psi$.

Since $E_{0}$ is type-definable, by saturation there exists a finite subset $\Psi_0\subseteq \Psi$ such that
\[E_0\subseteq  \bigcup_{(\varphi,\psi)\in \Psi_0} E_{\{\varphi,\psi\}}.\]

Assume that we cannot embed an infinite complete graph into $(p(G),E_{nalg})$.

\begin{claim}
For any $(\varphi,\psi)\in \Psi_0$,  $p\vdash \forall u(\varphi(u,x)\rightarrow\neg \psi(u,x))$.
\end{claim}
\begin{claimproof}
Choose any $\varphi,\psi\in \Psi_0$ and let $(a,b)\in E_{0}$ and $e$ be such that $\varphi(e,a)$, $\psi(e,b)$, $a\forkindep[e] b$ and $\tp (a/e)$ stationary.
Assume, toward a contradiction that there is some $e^\prime$ with $\varphi(e^\prime,b)$ and $\psi(e^\prime, b)$. Since both $\varphi(u,b)$ and $\psi(u,b)$ isolate a complete type over $b$ and are mutually consistent, they must be equivalent (i.e. define the same definable set).  So $\varphi(e,a)$ and $\varphi(e,b)$.

Let $\sigma$ be an automorphism satisfying $\sigma(a)=b$. Applying to the formulas above: $\varphi(e,b)$ and $\varphi(\sigma(e),b)$. Since $\varphi(u,b)$ isolates a complete type over $b$ there exists an automorphism $\tau$ fixing $b$ and mapping $\sigma(e)$ to $e$. Combining, $\tau\circ\sigma$ fixes $e$ and maps $a$ to $b$, i.e. $a\equiv_eb$.

By assumption $\tp(a/e)$ is stationary and $b\models \tp(a/e)|ea$ so we may construct a Morley sequence over $e$ starting with $a,b$. Since $a\E_0 b$ we get an infinite complete graph, contradicting our assumption.
\end{claimproof}

Note that each $E_{\{\varphi,\psi\}}$ defines a graph relation.
Let \[\theta(x,y)=\bigvee_{(\varphi,\psi)\in \Psi_0}  \exists e\left(\varphi(e,x)\wedge \psi(e,y)\right)\vee  \exists e\left(\varphi(e,y)\wedge \psi(e,x)\right).\]
Set $E_1=\{(a,b)\in E_{nalg}:\theta(a,b)\}$ and $E_2=\{(a,b)\in E_{nalg}:\neg \theta(a,b)\}$. Obviously, $E_{nalg}=E_1\cup E_2$ and both $E_1$ and $E_2$ are symmetric. If $\chi(p(G),E_2)\geq \aleph_1$ then by ($\dagger$) there exists $(a,b)\in E_0\cap E_2$, contradicting the choice of $\Psi_0$. Thus, by Lemma \ref{L:basic-prop-chi}(2), $\chi(p(G), E_1)\geq \aleph_1$.
Again by Lemma \ref{L:basic-prop-chi}(2), there is some $(\varphi,\psi)\in \Psi_0$ such that $\chi(G_{p,\varphi,\psi})\geq \aleph_1$.

$(3)$ Let $q_\varphi=\tp(e_1)$ for some $e_1\models \varphi(u,a)$ and some $a\models p$ and let $q_\psi=\tp(e_2)$ for some $e_2\models \psi(u,b)$ and some $b\models p$. Since $p$ is a complete type and $\varphi$ and $\psi$ each isolate a complete type it follows that $q_\varphi$ and $q_\psi$ do not depend on $a$, $b$, $e_1$ or $e_2$. 

As $(\varphi,\psi)\in \Psi$, there exist $(a,b)\in E_{nalg}$ and $e$ such that $\varphi(e,a)\wedge \psi(e,b)$ and $a\forkindep[e] b$. Consequently, $e\models q_{\varphi}$ and $e\models q_{\psi}$ and hence $q_\varphi=q_\psi$.
Since $e\in \acl(a)$, 
\[\Ur(a/e)+\Ur(e)=\Ur(a).\] 

If $\Ur(e)=0$ then $a\forkindep e$ so by transitivity of forking $a\forkindep b$, but then we may embed an infinite complete graph as in Lemma \ref{L:must-be-U2}, which contradicts $(2)$. If $\Ur(a/e)=0$ then $a\in \acl(e)\subseteq \acl(b)$ so $a$ and $b$ are interalgebraic, contradiction (see above Lemma \ref{L:must-be-U2}).

$(4)$  Note that $R$ defines a graph, i.e. it is irreflexive by $(2)$. Let $n$ be $|\varphi(G,a)|$ and $m$ be $|\psi(G,a)|$ for some (any) $a\models p$. For any $a\models p$ choose enumerations $\varphi(G,a)=\{e_i(a):i<n\}$ 
and $\psi(G,a)=\{e_i^\prime(a):i<m\}$. 

For $i<n$ and $j<m$ let $H_{i,j}=\{(e_i(a),e_j^\prime(a)):a\models p\}.$

We define an edge relation on $H_{i,j}=\{(e_i(a),e_j^\prime(a)):a\models p\}$ (for $i<n, j<m$) as follows: $(e_i(a),e_j^\prime(a))$ is connected to $(e_i(b),e_j^\prime(b))$ if and only if $e_i(b)=e_j^\prime(a)$ or $e_i(a)=e_j^\prime(b)$. Note that $e_i(a)\neq e_j^\prime (a)$ for all $a\models p$ and $i<n,\, j<m$ by $(2)$, hence this relation is irreflexive.

\begin{claim}
There exist $i_0<n$ and $j_0<m$ such that $\chi(H_{i_0,j_0})\geq \aleph_1$.
\end{claim}
\begin{claimproof}
Assume that for all $i<n,\, j<m$, $H_{i,j}$ is countably colorable, say by the coloring function $c_{i,j}:H_{i,j}\to \aleph_0$. We claim that this entails that $G_{p,\varphi,\psi}$ is countably colorable, which would give a contradiction to choice of $(\varphi,\psi)$.

We define a coloring $c:G_{p,\varphi,\psi}\to (\aleph_0)^{n\times m}$ by $c(a)(i,j)=c_{i,j}(e_i(a),e_j^\prime(a))$. The contradiction will follow if we show that this is a legal coloring. Let $(a,b)\in E_{\varphi,\psi}$ ($(a,b)\in E_{\psi,\varphi}$ is similar). Thus there exists some $e\models \varphi(u,a)\wedge \psi(u,b)$. Consequently, $e=e_i(a)=e_j^\prime(b)$ for some $i<n,j<m$, so \[c_{i,j}(e_i(a),e_j^\prime(a))\neq c_{i,j}(e_i(b),e_j^\prime(b)),\]
and $c(a)\neq c(b)$.
\end{claimproof}

Now, assume that $\chi(H_q)\leq \aleph_0$ and let $c:q(G)\to\aleph_0$ be a coloring. We define $f:H_{i_0,j_0}\to\aleph_0\times \aleph_0$ by $f(e_{i_0}(a),e_{j_0}^\prime(a))=\left(c(e_{i_0}(a)),c(e_{j_0}^\prime(a))\right)$. This gives a legal coloring of $H_{i_0,j_0}$ using countably many colors, and we reach a contradiction:  assume with out loss of generality that $(e_{i_0}(a),e^\prime_{j_0}(a)),(e_{i_0}(b),e_{j_0}^\prime(b))\in H_{i_0,j_0}$ with $e^\prime_{j_0}(a)=e_{i_0}(b)$. Since $e_{i_0}(a)\R e_{j_0}^\prime(a)$, $c(e_{i_0}(a))\neq c(e_{j_0}^\prime(a))=c(e_{i_0}(b))$ and thus $f(e_{i_0}(a),e^\prime_{j_0}(a))\neq f(e_{i_0}(b),e_{j_0}^\prime(b))$, as needed.

\end{proof}

\begin{remark}
We remark that if $\tp(a/e)$ is stationary and $e^\prime\models \tp(e/a)$ then $\tp(a/e^\prime)$ is also stationary. 
\end{remark}

The procedure outlined in the items of Proposition \ref{P:main-prop-finite-u-rank} supplies, under some assumptions, a graph, with uncountable chromatic number, concentrated on a type of lower $\Ur$-rank than the one we started with. This hints that some induction procedure may be possible (at least for one-based theories). We will not pursue this further now. For now we concentrate on graphs of at most $\Ur$-rank $2$.

The following is an easy exercise in stability theory.
\begin{lemma}\label{L:stationarity}
Let $T$ be a stable theory. Assume that $A\forkindep[C] B$, $A^\prime\forkindep[C] B^\prime$, $B\equiv_C B^\prime$, $A\equiv_C A^\prime$ and $\tp(A/C)$ stationary. Then $AB\equiv_C A^\prime B^\prime$.
\end{lemma}
%
%
%
%

\begin{theorem}\label{T:Uleq2}
Let $G=(V,E)$ be an $\omega$-stable graph with $\chi(G)\geq\aleph_1$. If $\Ur(G)\leq 2$ then $G$ contains all finite subgraphs of $\Sh_n(\omega)$ for some $n\leq 2$. 
\end{theorem}
\begin{proof}
We may assume that $G$ is saturated (in particular $\aleph_1$-saturated) and we may also work in $G^{eq}$. Fix some countable $G_0\prec G$ and add constants for it (so every type over $\emptyset$ is stationary). By $\omega$-stability and Lemma \ref{L:basic-prop-chi}(1) there is some type $p\in S_1(\emptyset)$ such that $\chi(G_p)\geq \aleph_1$ with $G_p=(p(G),E\restriction p(G))$.
We are now in the situation of Assumption $\diamondsuit$.

If $\Ur(p)=0$ then $p$ is algebraic (even realized), contradicting $\chi(G_p)\geq \aleph_1$. If $\Ur(p)=1$ then we may embed an infinite complete graph by Lemma \ref{L:must-be-U2}(3). 

We may thus assume that $\Ur(p)=2$ and that $G$ does not contain an infinite complete graph. 

Let $E_0\subseteq E$ be the type-definable set from Lemma \ref{L:E_0-for-U-2} and $\varphi$, $\psi,G_{p,\varphi,\psi}$, $q$, $R$ and $H_q$ be as supplied by Proposition \ref{P:main-prop-finite-u-rank} with respect to $\E_0$ and $p$ (so necessarily $\Ur(q)=1$). Noting that Assumption $\diamondsuit$ is true for $H_q$, we may apply Lemma \ref{L:must-be-U2}(3) and thus there exist a Morley sequence $\langle e_i: i<\omega\rangle$ such that $e_i \R e_j$ for all $i\neq j$. 

Since $e_0\R e_1$ then there is some $a_{0,1}\models p$ such that, without loss of generality, $\varphi(e_0,a_{0,1})\wedge \psi(e_1,a_{0,1})$. For any $i<j<\omega$ let $a_{i,j}$ be such that $a_{i,j}e_ie_j\equiv a_{0,1}e_0e_1$. 

Note that for every $i<j<\omega$, $a_{i,j}\in \acl(e_i,e_j)$. Indeed, by Lascar's equality
$\Ur(a_{i,j}/e_ie_j)+\Ur(e_i/e_j)=\Ur(a_{i,j}e_i/e_j)$ and since $e_i\in\acl(a_{i,j})$ the right hand side is also equal to $\Ur(a_{i,j}/e_j)$. Now we note that 
$\Ur(a_{i,j}/e_j)+\Ur(e_j)=\Ur(a_{i,j}e_j)$, but as before $e_j\in \acl(a_{i,j})$ so the right hand side is equal to $2$ and since $\Ur(e_j)=1$ we conclude that $\Ur(a_{i,j}/e_j)=1$. As $e_i\forkindep e_j$ we have that $\Ur(e_i/e_j)=1$ as well so we combine everything and get that $\Ur(a_{i,j}/e_ie_j)=0$.

Define a map $f:\Sh_2(\omega)\to G_{p,\varphi,\psi}$ by $(i,j)\mapsto a_{i,j}$. We claim that this is an injective graph homomorphism.

For $(i,j)\neq (i^\prime,j^\prime)\in \Sh_2(\omega)$, $a_{i,j}$ and $a_{i^\prime,j^\prime}$ are not interalgebraic and in particular $f$ is injective. Indeed, assume $\acl(a_{i,j})=\acl(a_{i^\prime,j^\prime})$. Since $(i,j)\neq (i^\prime,j^\prime)$,  $|\{i,j,i^\prime,j^\prime\}|\geq 3$, and we assume that $i\neq j,i^\prime,j^\prime$ (the other cases are similar). On the other hand, \[e_i\in \acl(a_{i,j})=\acl(a_{i^\prime,j^\prime})\subseteq \acl(e_{i^\prime},e_{j^\prime}),\]
contradicting indiscernibility.

$f$ is a graph homomorphism. Let $(i,j),(j,k)\in \Sh_2(\omega)$, so $i<j<k<\omega$. As $a_{i,j}$ and $a_{j,k}$ are not interalgebraic, necessarily $a_{i,j}\forkindep[e_j] a_{j,k}$ for otherwise
\[ \Ur(a_{i,j}/e_ja_{j,k})<\Ur(a_{i,j}/e_j)=1\]
and then $a_{i,j}\in \acl(e_ja_{j,k})\subseteq \acl(a_{j,k})$.

By the choice of $(\varphi,\psi)$ in Proposition \ref{P:main-prop-finite-u-rank}, we may find $a,b\models p$ and $e\models q$ with $a \E b$, $a\forkindep[e] b$, $\tp(a/e)$ stationary, $\varphi(e,a)$ and $\psi(e,b)$. By applying an automorphsim mapping $e_j$ to $e$ we may assume $e_j=e$. Let $\sigma$ be an automorphism mapping $a$ to $a_{j,k}$, thus $ae\equiv a_{j,k}\sigma(e)$ and $\sigma(e)\models \varphi(u,a_{j,k})$. Applying now an automorphism mapping $\sigma(e)$ to $e$ but fixing $a_{j,k}$ we conclude that $a\equiv_e a_{j,k}$ and similarly $b\equiv_e a_{i,j}$. Since $\tp(a/e)$ is stationary, by Lemma \ref{L:stationarity}, $ab\equiv a_{j,k}a_{i,j}$ so $a_{i,j}\E a_{j,k}$ as well.
\end{proof}

\bibliographystyle{alpha}
\bibliography{1196}

\end{document}